\documentclass[5p,preprint]{elsarticle}

\usepackage
{
    graphicx,
    amssymb,
    amsmath,
    amsthm,
    dsfont, 
    xcolor,
    mathtools,
    nicefrac,
    array,
    enumerate,
    color
}

\usepackage[notrig]{physics}
\usepackage[utf8]{inputenc}
\usepackage{todonotes}

\usepackage[pdffitwindow=false,
            plainpages=false,
            pdfpagelabels=true,
            pdfpagemode=UseOutlines,
            pdfpagelayout=SinglePage,
            bookmarks=false,
            colorlinks=true,
            hyperfootnotes=false,
            linkcolor=blue,
            urlcolor=blue!30!black,
            citecolor=green!50!black]{hyperref}

\usepackage[bf,normal]{caption}

\graphicspath{{Figures/}}

\newcommand{\R}{\mathbb{R}}                                     
\newcommand{\pd}[2]{\frac{\partial#1}{\partial#2}}              
\newcommand{\ddt}{\tfrac{\text{\normalfont d}}{\text{\normalfont d}t}} 

\newcolumntype{C}[1]{>{\centering\let\newline\\\arraybackslash\hspace{0pt}}m{#1}}

\newtheorem{theorem}{Theorem}[section]

\newtheorem{lemma}[theorem]{Lemma}

\theoremstyle{definition}
\newtheorem{example}[theorem]{Example}
\newtheorem{remark}[theorem]{Remark}

\newcommand{\cF}{\mathcal{F}}

\newcommand{\eps}{\varepsilon}

\newcommand{\setdef}[2]{\left\{\, #1 \left|\, \vphantom{#1} #2\right.\right\}}

\DeclareOldFontCommand{\rm}{\normalfont\rmfamily}{\mathrm}

\makeatletter
\renewcommand*\env@matrix[1][*\c@MaxMatrixCols c]{%
  \hskip -\arraycolsep
  \let\@ifnextchar\new@ifnextchar
  \array{#1}}
\makeatother

\mathtoolsset{centercolon} 

\allowdisplaybreaks

\DeclareCaptionLabelFormat{period}{#1~#2}
\makeatletter
\def\blfootnote{\gdef\@thefnmark{}\@footnotetext}
\makeatother

\newenvironment{smallpmatrix}
{\left(\begin{smallmatrix}}
{\end{smallmatrix}\right)}

{\left[\begin{smallmatrix}}
{\end{smallmatrix}\right]}

\begin{document}

\begin{frontmatter}

\title{Funnel Control for Langevin Dynamics}

\author[1]{Thomas Berger}\ead{thomas.berger@math.upb.de}
\author[2]{Feliks Nüske}\ead{nueske@mpi-magdeburg.mpg.de}
\address[1]{Universit\"at Paderborn, Institut f\"ur Mathematik, Warburger Str.~100, 33098~Paderborn, Germany}
\address[2]{Max Planck Institute for Dynamics of Complex Technical Systems, Magdeburg, Germany}


\begin{abstract}
We study tracking control for stochastic differential equations of Langevin type and describe a new conceptual approach to the sampling problem for those systems. The objective is to guarantee the evolution of the mean value in a prescribed performance funnel around a given sufficiently smooth reference signal. To achieve this objective we design a novel funnel controller and show its feasibility under certain structural conditions on the potential energy. The control design does not require any specific knowledge of the shape of the potential energy. We illustrate the results by a numerical simulation for a double-well potential.
\end{abstract}

\end{frontmatter}

\section{Introduction}
\label{sec:introduction}
In this paper, we study the applicability of the funnel controller to stochastic differential equations (SDEs) of Langevin type. The funnel controller was developed in the seminal work~\cite{IlchRyan02b} (see also the recent survey in~\cite{BergIlch21}) and is a low-complexity model-free output-error feedback of high-gain type. Since it only requires knowledge of some structural properties of the system, but not of any specific system parameters, the funnel controller is inherently robust and hence suitable for applications in highly uncertain plants or environments. It proved advantageous in a variety of applications such as control of industrial servo-systems~\cite{Hack17}, underactuated multibody systems~\cite{BergDrue21}, peak inspiratory pressure~\cite{PompWeye15} and adaptive cruise control~\cite{BergRaue20}.

SDEs are routinely used to model dynamical systems subject to uncertainties all across the natural and engineering sciences~\cite{Okse03,Evans2012}, with applications including financial markets, atmospheric dynamics, and molecular dynamics. A specific class of SDEs are Langevin dynamics, where the drift is given as the negative gradient of a scalar energy function, while the diffusion is constant. Langevin dynamics itself have been used in many different contexts, but perhaps most prominently, it constitutes a popular dynamical model for molecular systems. The main reason is that, under mild assumptions, its associated invariant measure is the Boltzmann distribution, which is an object of central importance in statistical physics. Therefore, long trajectories of Langevin dynamics can be used to estimate expectation values with respect to the Boltzmann distribution. It should be noted, however, that dynamical quantities derived from Langevin dynamics have also attracted significant attention in molecular modeling, see for example~\cite{LeliStol16} for a mathematical review of this topic. Langevin dynamics can also be obtained as the high-friction limit of underdamped Langevin dynamics, which is frequently used to model molecular systems as Hamiltonian dynamics coupled to a stochastic environment~\cite{Stoltz2010}.

A major impediment to the use of Langevin dynamics in applications is metastability, meaning that due to the presence of multiple local minima of the energy, separated by sharp barriers, the dynamics tends to spend long times oscillating around the same configuration, severely slowing down the process of sampling the Boltzmann measure. To circumvent this so called \emph{sampling problem}, a wide variety of different numerical approaches have been developed~\cite{Rohrdanz2013,Sidky2020}, but the problem remains essentially open to this day. In the literature, approaches to the sampling problem based on optimal control have also been developed~\cite{Hartmann2012}, but a priori knowledge of the system is often required in order to achieve satisfactory performance. Against this backdrop, we perform a theoretical study of the funnel controller in this context, which 
does not require any information about the system parameters. Our main result provides \textit{structural} conditions on the system and design parameters such that solutions of the controlled SDE are guaranteed to exist and to achieve the control objective. ``Structural'' means that for fixed controller design parameters the funnel controller achieves the objective for a whole class of systems~-- this class is either empty or contains an open ball. Although many more questions remain open, this result provides a theoretical basis for a new angle to tackle the long-standing sampling problem for metastable systems.

Control of stochastic systems under state or tracking error constraints is considered in other works in the literature~\cite{SuiChen21,SuiShao15,ZhanXia18}. However, all of these works assume a special feedback structure for the system and the output is a stochastic process~-- for which it is not clear how it can be measured. Tracking with prescribed performance for the mean value of SDEs has not been considered so far.

The present paper is organized as follows. In Section~\ref{sec:prerequisites} we present a precise statement of the considered tracking problem, including the assumptions on the considered class of SDEs. Some existence and uniqueness results for SDEs are recalled in Section~\ref{Sec:LangDyn} and a differential equation for certain mean values is derived. The main result on tracking by funnel control for SDEs is stated and proved in Section~\ref{sec:feasibility}. The latter means to show the existence and uniqueness of a solution to a time-varying nonlinear SDE with a singularity on the right-hand side. This result is illustrated by a simulation of a double-well potential in Section~\ref{sec:examples}. It is rigorously shown that this example satisfies the assumptions of the SDE system class. The paper concludes with Section~\ref{sec:Conclusion}.

\paragraph{Nomenclature} In the following let $\R$ denote the real numbers, $\R_{\ge 0} = [0,\infty)$ and $\R^{m\times n}$ the set of matrices of size $m\times n$. $L^\infty(I,\R^n)$ is the Lebesgue space of measurable, essentially bounded functions $f:I\to\R^n$, where $I\subseteq \R$ is some interval, with norm $\|\cdot\|_\infty$. $W^{k,\infty}(I,\R^n)$ is the Sobolev space of all functions  $f:I\to\R^n$ with $k$-th order weak derivative $f^{(k)}$ and $f, \dot f, \ldots, f^{(k)}\in L^\infty(I,\R^n)$. $C^k(M,\R^n)$ is the set of $k$-times continuously differentiable functions $f:M\to\R^n$, where $M\subseteq \R^m$. By $\mathbb{E}[X]$ we denote the mean value of a random variable~$X:\Omega\to\R$, where $\Omega$ denotes the sample space of a probability space $(\Omega,\cF,P)$.

\section{Problem statement}
\label{sec:prerequisites}

\subsection{System class}
We consider the controlled stochastic process with dynamics given by the stochastic differential equation (SDE, cf.~\cite[Sec.~11]{Okse03})
\begin{equation}
\label{eq:sde}
{\rm d} X_t = -\big(\nabla V(X_t) +  A(X_t - u(t)) \big)\,{\rm d}t + \sqrt{2}\,{\rm d}B_t,
\end{equation}
where $X_t:\Omega\to\R^d$, $t\ge 0$, are random vectors and $\Omega$ is the sample space of a probability space $(\Omega,\cF,P)$. $(B_t)_{t\ge0}$ is $d$-dimensional Brownian motion (a Wiener process with zero mean value and unit variance), $V : \mathbb{R}^d \mapsto \mathbb{R}$ is the \emph{potential energy} and $A \in \mathbb{R}^{d \times d}$ is a state-independent symmetric positive definite matrix. The function $u:\R_{\ge 0}\to\R^d$ is the control input.

The process described by Eq.~\eqref{eq:sde} is known as \emph{Langevin dynamics} defined by the potential $V$, subject to an additional forcing due to the gradient of a time-dependent quadratic \emph{biasing potential}
\[V_b(x, t) = \frac{1}{2}(x - u(t))^\top A (x - u(t)). \]
The matrix~$A$ can be viewed as a design parameter for this biasing potential and will later be tuned to achieve feasibility of the to-be-designed feedback controller.
We make the following assumptions on the potential~$V$ and the matrix~$A$.
\begin{enumerate}
\item[$\mathbf{(A1)}$] $V\in C^2(\R^d,\R)$, $V(x)\ge 0$ for all $x\in\R^n$, and $\nabla V$ is globally Lipschitz continuous,
\item[$\mathbf{(A2)}$] $\exists \, c_1, c_2>0\ \forall\, x\in\R^d:\ \Delta V(x) + \tr A - \|\nabla V(x) + Ax\|^2 \le -c_1 (V(x) + \tfrac12 x^\top A x) + c_2,$
\item[$\mathbf{(A3)}$] $\exists\, c_3, c_4>0\ \forall\, x\in\R^d:\ \|\nabla V(x)\| \le c_3\big( V(x) + \tfrac12 x^\top A x\big) + c_4$.
\end{enumerate}
Assumption~(A1) is common and ensures existence of a solution to the uncontrolled equation (i.e., $u=0$), cf.\ also Section~\ref{Sec:LangDyn}. Assumption~(A2) resembles a global version of the growth condition in~\cite[Thm.~2.5]{LeliStol16}, which is used there to derive a Poincar\'{e} inequality. Assumption~(A3) essentially means that~$V$ exhibits at most exponential growth, where the growth rate may even depend on~$A$. The assumptions~(A1)--(A3) are generally easy to satisfy and always hold for quadratic potentials~$V$ as shown in the following example.

\begin{example}\label{Ex:quad-pot}
Let $A,S\in\R^{d\times d}$ be symmetric and positive definite, $b\in\R^d$ and $f\in\R$ such that
\[
    f \ge \tfrac12 b^\top S^{-1} b
\]
and consider the potential
\[
    V:\R^d\to \R,\ x\mapsto \tfrac12 x^\top S x + b^\top x + f.
\]
We show that (A1)--(A3) are satisfied and calculate the constants $c_1,\ldots,c_4$ explicitly. To this end, let $\lambda_{\min}(M), \lambda_{\max}(M)$ denote the minimal and maximal eigenvalue of a matrix $M\in\R^{d\times d}$, resp. Since $V$ attains its minimum at $x^* = -S^{-1} b$ it is clear that (A1) holds, and for (A2) we calculate
\begin{align*}
  & \Delta V(x) + \tr A - \|\nabla V(x) + Ax\|^2 \\
  & = \tr S + \tr A - \|(S+A)x + b\|^2\\
  &= \tr S + \tr A - x^\top (S+A)^2 x - 2 b^\top (S+A) x - \|b\|^2 \\
  &\le \tr S + \tr A - \lambda_{\min}(S+A) x^\top (S+A) x \\
  &\quad + b^\top (c_1 I - 2(S+A)) x - c_1 b^\top x - \|b\|^2\\
  &\le \tr S + \tr A - \lambda_{\min}(S+A) x^\top (S+A) x \\
  &\quad +  \|c_1 I - 2(S+A)\|\, \|b\|\, \|x\| - c_1 b^\top x - \|b\|^2\\
  &= \tr S + \tr A - \lambda_{\min}(S+A) x^\top (S+A) x \\
  &\quad +  \big(2\lambda_{\max}(S+A) - c_1\big) \|b\|\, \|x\| - c_1 b^\top x - \|b\|^2\\
  &\le \tr S \!+\! \tr A \!-\! \lambda_{\min}(S+A) x^\top (S \!+\!A) x \\
    &\quad + \frac{\eps}{2} \|x\|^2 \!+\! \frac{\big(2\lambda_{\max}(S \!+\!A) \!-\! c_1\big)^2 \|b\|^2}{2\eps} \!-\! c_1 b^\top x - \|b\|^2\\
  &\le \tr S \!+\! \tr A \!-\! \left(\lambda_{\min}(S\!+\!A) \!-\! \frac{\eps}{2\lambda_{\min}(S\!+\!A)}\right) x^\top (S\!+\!A) x \\
  &\quad + \left(\frac{\big(2\lambda_{\max}(S+A) - c_1\big)^2}{2\eps} -1\right)\|b\|^2 - c_1 b^\top x \\
  &\overset{\eps = \lambda_{\min}(S+A)^2}{\le} \tr S + \tr A - \tfrac12 \lambda_{\min}(S+A) x^\top (S+A) x \\
  &\quad + \left( \frac{\big(2\lambda_{\max}(S+A) - c_1\big)^2}{2\lambda_{\min}(S+A)^2} - 1\right)\|b\|^2 - c_1 b^\top x \\
  &\le -\frac{c_1}{2}\left( x^\top (S+A) x + 2 b^\top x + 2f\right) + c_2\\
  &=-c_1 (V(x) + \tfrac12 x^\top A x) + c_2
\end{align*}
for $c_1 = \lambda_{\min}(S+A)$ and
\begin{align*}
    c_2 &= \tr S \!+\! \tr A \!+\! \left(\frac{\big(2\lambda_{\max}(S\!+\!A) \!-\! \lambda_{\min}(S\!+\!A)\big)^2}{2\lambda_{\min}(S\!+\!A)^2} \!-\! 1\right)\|b\|^2 \\
    &\quad + \lambda_{\min}(S+A) f,
\end{align*}
where we have used that for $c_1< 2 \lambda_{\min}(S+A)$ the matrix $2(S+A)-c_1 I$ is symmetric and positive definite and hence its spectral norm is given by its maximal eigenvalue $2\lambda_{\max}(S+A) - c_1$. For (A3) we calculate
\begin{align*}
  & \|\nabla V(x)\| = \|Sx+b\| \\
  &\le \lambda_{\max}(S) \|x\| + \|b\| + c_3 \|b\|\, \|x\| + c_3 b^\top x\\
  &= -\frac{c_3}{2} \lambda_{\min}(S+A) \left(\|x\| - \frac{\lambda_{\max}(S) + c_3 \|b\|}{c_3 \lambda_{\min}(S+A)}\right)^2 \\
  &\quad + \frac{c_3}{2} \lambda_{\min}(S+A) \|x\|^2 + \frac{(\lambda_{\max}(S) + c_3 \|b\|)^2}{2 c_3 \lambda_{\min}(S+A)}\\
  &\quad  + \|b\| + c_3 b^\top x\\
  &\le \frac{c_3}{2} x^\top (S+A) x + c_3 b^\top x + c_3 f + c_4\\
  &= c_3\big( V(x) + \tfrac12 x^\top A x\big) + c_4
\end{align*}
for arbitrary $c_3>0$ and $c_4 = \frac{(\lambda_{\max}(S) + c_3 \|b\|)^2}{2 c_3 \lambda_{\min}(S+A)} + \|b\| - c_3 f$.
\end{example}

We associate an output function $y:\R_{\ge 0}\to \R^d$ with~\eqref{eq:sde}, for which we seek to achieve a desired behavior and for which instantaneous measurements are assumed to be available. In virtue of~\cite{AnnuBorz10}, a canonical choice for the output is the mean value $\mathbb{E}[X_t]$, which ``is omnipresent in almost all stochastic optimal control problems considered in the scientific literature''.  Therefore, we define
\begin{equation}\label{eq:output}
  y(t) = \mathbb{E}[X_t] = \begin{pmatrix} \mathbb{E}[(X_t)_1] \\ \vdots\\ \mathbb{E}[(X_t)_d]  \end{pmatrix}.
\end{equation}
In practice, it is hard to calculate the corresponding integrals $\mathbb{E}[(X_t)_i]$ exactly; but they may be approximated by data-driven methods such as Monte Carlo integration.

\subsection{Control objective}\label{Ssec:ContrObj}

The objective is to design an output error feedback strategy $u(t) = F(t,e(t))$, where $e(t)= y(t) - y_{\rm ref}(t)$ for some reference trajectory $y_{\rm ref}\in W^{1,\infty}(\R_{\ge 0};\R^d)$, such that in the closed-loop system the tracking error $e(t)$ evolves within a prescribed performance funnel
\begin{equation*}
\mathcal{F}_{\psi} := \setdef{(t,e)\in\R_{\ge 0} \times\R^d}{\|e\| < \psi(t)},\label{eq:perf_funnel}
\end{equation*}
which is determined by a function~$\psi$ belonging to
\begin{equation*}
\Psi:= \setdef{ \psi\in W^{1,\infty}(\R_{\ge 0};\R)}{\!\!\begin{array}{l} \text{$\psi(t)>0$ for all $t>0$},\\ \liminf_{t\to\infty} \psi(t)>0\end{array}\!\!}.
\label{eq:Psi}
\end{equation*}
By the properties of~$\Psi$ there exists $\lambda>0$ such that $\psi(t)\ge \lambda$ for all $t\ge 0$. Therefore, practical tracking with arbitrary small accuracy $\lambda>0$ can be achieved. The situation is depicted in Fig.~\ref{Fig:funnel}.

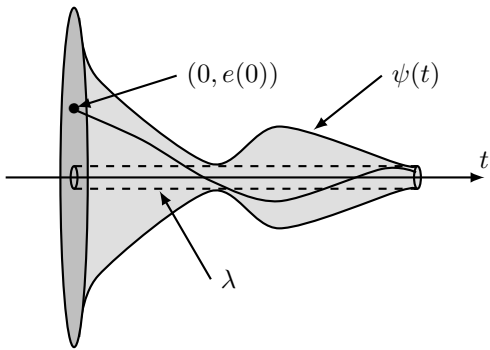
\begin{figure}[h]
\begin{center}
\begin{tikzpicture}[scale=0.45]
\tikzset{>=latex}
  \filldraw[color=gray!25] plot[smooth] coordinates {(0.15,4.7)(0.7,2.9)(4,0.4)(6,1.5)(9.5,0.4)(10,0.333)(10.01,0.331)(10.041,0.3) (10.041,-0.3)(10.01,-0.331)(10,-0.333)(9.5,-0.4)(6,-1.5)(4,-0.4)(0.7,-2.9)(0.15,-4.7)};
  \draw[thick] plot[smooth] coordinates {(0.15,4.7)(0.7,2.9)(4,0.4)(6,1.5)(9.5,0.4)(10,0.333)(10.01,0.331)(10.041,0.3)};
  \draw[thick] plot[smooth] coordinates {(10.041,-0.3)(10.01,-0.331)(10,-0.333)(9.5,-0.4)(6,-1.5)(4,-0.4)(0.7,-2.9)(0.15,-4.7)};
  \draw[thick,fill=lightgray] (0,0) ellipse (0.4 and 5);
  \draw[thick] (0,0) ellipse (0.1 and 0.333);
  \draw[thick,fill=gray!25] (10.041,0) ellipse (0.1 and 0.333);
  \draw[thick] plot[smooth] coordinates {(0,2)(2,1.1)(4,-0.1)(6,-0.7)(9,0.25)(10,0.15)};
  \draw[thick,->] (-2,0)--(12,0) node[right,above]{\normalsize$t$};
  \draw[thick,dashed](0,0.333)--(10,0.333);
  \draw[thick,dashed](0,-0.333)--(10,-0.333);
  \node [black] at (0,2) {\textbullet};
  \draw[->,thick](4,-3)node[right]{\normalsize$\lambda$}--(2.5,-0.4);
  \draw[->,thick](3,3)node[right]{\normalsize$(0,e(0))$}--(0.07,2.07);
  \draw[->,thick](9,3)node[right]{\normalsize$\psi(t)$}--(7,1.4);
\end{tikzpicture}
\end{center}
\caption{Error evolution in a funnel $\mathcal F_{\psi}$ with boundary $\psi(t)$.}
\label{Fig:funnel}
\end{figure}

It is important to note that the function $\psi\in\Psi$ is a design parameter in the control law (stated in Section~\ref{sec:feasibility}) and its choice is up to the designer. Typically, the constraints on the tracking error are due to the specific application, which hence indicates suitable choices for~$\psi$. Although the funnel boundary does not need to be monotonically decreasing in general, it is often convenient to choose a monotone $\psi$. However, widening the funnel over some later time interval may help to reduce the maximal control input and improve the controller performance, for instance in the presence of strongly varying reference signals or periodic disturbances. Typical choices for funnel boundaries are outlined in~\cite[Sec.~3.2]{Ilch13}.

\section{Solutions of the Controlled Langevin Equation}\label{Sec:LangDyn}


First we recall under which conditions on the potential~$V$ and the control input~$u$ the SDE~\eqref{eq:sde} has a unique solution for an admissible initial condition $X_0 = Z$. By an \textit{admissible initial condition} we mean a random variable $Z$, which is independent of the $\sigma$-algebra generated by $B_s$, $s\ge 0$, and such that $\mathbb{E}[\|Z\|^2]< \infty$. By a \textit{solution} of~\eqref{eq:sde} with $X_0 = Z$ for a measurable function $u:\R_{\ge 0}\to\R^d$, we mean a $t$-continuous stochastic process $(X_t)_{t\ge 0}$, adapted to the filtration $\cF_t^Z$ generated by~$Z$ and $B_s$, $s\le t$, which satisfies $\mathbb{E}\left[ \int_0^T \|X_t\|^2 {\rm d}t\right] < \infty$ for all $T>0$, and solves the stochastic integral equation
\[
    X_t = Z - \int_0^t \big(\nabla V(X_s) +  A (X_s-u(s))\big) {\rm d}s + \int_0^t \sqrt{2} \, {\rm d} B_s
\]
for $t\ge 0$. The existence and uniqueness result for the SDE~\eqref{eq:sde} is given in the following, and is a consequence of~\cite[Thm.~5.2.1]{Okse03}.

\begin{lemma}\label{Lem:SDE-sln}
Let $u\in L^\infty(\R_{\ge 0},\R^d)$ and $V\in C^1(\R^d,\R)$ such that $\nabla V$ is globally Lipschitz continuous. Then for any admissible initial condition $X_0 = Z$, the SDE~\eqref{eq:sde} has a unique solution.
\end{lemma}

Essentially the same result holds for the slightly modified SDE
\begin{equation}
\label{eq:sde-mod}
{\rm d} X_t = -\big(\nabla V(X_t) +  A \big(X_t + d(t,\mathbb{E}[X_t])\big) \big)\,{\rm d}t + \sqrt{2}\,{\rm d}B_t,
\end{equation}
where $d:\R_{\ge 0}\times \R^d\to\R^d$ is a measurable and essentially bounded function. Solutions of~\eqref{eq:sde-mod} are defined analogously to~\eqref{eq:sde}. The proof is a straightforward modification of that of~\cite[Thm.~5.2.1]{Okse03}, using the boundedness of~$d$.

\begin{lemma}\label{Lem:SDE-mod-sln}
Let $d\in L^\infty(\R_{\ge 0}\times \R^d,\R^d)$ and $V\in C^1(\R^d,\R)$ such that $\nabla V$ is globally Lipschitz continuous. Then for any admissible initial condition $X_0 = Z$, the SDE~\eqref{eq:sde-mod} has a unique solution.
\end{lemma}

Next, we recapitulate how we may derive an expression for the derivative of $\mathbb{E}[\phi(X_t)]$,
where $\phi \in C^2(\mathbb{R}^d,\R)$, by utilizing the SDE~\eqref{eq:sde}. Set $Y_t := \phi(X_t)$ for $t\ge 0$. From the multidimensional It\^{o} formula (see e.g.~\cite[Thm.~4.2.1]{Okse03}) it follows that
\[
    {\rm d} Y_t = \nabla \phi(X_t)^\top {\rm d} X_t + \tfrac12 ({\rm d} X_t)^\top \nabla^2 \phi(X_t) {\rm d} X_t.
\]
Using the standard rules (see e.g.~\cite[Thms.~4.1.2~\&~4.2.1]{Okse03})
\[
     {\rm d}t \cdot  {\rm d}t =  {\rm d}t \cdot  ({\rm d} B_t)_i = ({\rm d} B_t)_i \cdot {\rm d}t = 0,\quad ({\rm d} B_t)_i \cdot ({\rm d} B_t)_j = \delta_{ij} {\rm d}t
\]
for $i,j=1,\ldots,d,$ we may derive that
\[
    {\rm d} Y_t = (\mathcal{L}^{u(t)} \phi)(X_t) {\rm d}t + \sqrt{2} \nabla \phi(X_t)^\top {\rm d} B_t,
\]
where, for some $v\in\R^d$, $\mathcal{L}^{v}$ is the second-order linear differential operator
\begin{equation}
\label{eq:def_generator}
(\mathcal{L}^{v} \phi)(x) = -\big(\nabla V(x) + A (x - v)\big)^\top \nabla \phi(x) + \Delta \phi(x),
\end{equation}
where $x\in\R^d$. Written in integral form we have
\[
    Y_t = Y_0 + \int_0^t (\mathcal{L}^{u(t)} \phi)(X_s) {\rm d}s + \sqrt{2} \int_0^t \nabla \phi(X_s)^\top {\rm d} B_s,
\]
and by~\cite[Thm.~3.2.1]{Okse03} we obtain the implication
\begin{equation}\label{eq:impl-phi}
\begin{aligned}
    &\mathbb{E}\left[ \int_0^t \left(\frac{\partial \phi}{\partial x_i}(X_s)\right)^2 {\rm d}s\right] < \infty\\
    &\implies \quad \mathbb{E}\left[ \int_0^t \frac{\partial \phi}{\partial x_i}(X_s) {\rm d} (B_s)_i\right] = 0
    \end{aligned}
\end{equation}
for $i=1,\ldots,d$. Then we have
\[
    \mathbb{E}[Y_t] = \mathbb{E}[Y_0] + \int_0^t \mathbb{E}[(\mathcal{L}^{u(s)} \phi)(X_s)] {\rm d}s
\]
and taking the derivative yields
\begin{equation}
\label{eq:time_derivative_obs}
\ddt  \mathbb{E}[\phi(X_t)] = \mathbb{E}[\mathcal{L}^{u(t)} \phi(X_t)].
\end{equation}
The condition on the left hand side of the implication~\eqref{eq:impl-phi} is satisfied for all $t\ge 0$ and all $i=1,\ldots,d$, if $\nabla \phi$ is globally Lipschitz continuous, which can be seen as follows: Since $\nabla \phi$ is in particular linearly bounded, i.e., $\|\nabla \phi(x)\| \le c (1 +\|x\|)$ for all $x\in\R^d$, it follows that
\[
    \left(\tfrac{\partial \phi}{\partial x_i}(x)\right)^2 \!\le \nabla \phi(x)^\top \nabla \phi(x) \le c^2 (1 + \|x\|)^2 \le 2 c^2 (1 +  \|x\|^2),
\]
where we have used $2a \le 1 + a^2$ for any $a\ge 0$. Therefore,
\begin{align*}
    \mathbb{E}\left[ \int_0^t \left(\frac{\partial \phi}{\partial x_i}(X_s)\right)^2 {\rm d}s\right] \!\le\! 2 c^2 \left(t \!+\! \mathbb{E}\left[ \int_0^t \|X_s\|^2 {\rm d}s\right] \right) \!<\! \infty,
\end{align*}
using that $\mathbb{E}\left[ \int_0^t \|X_s\|^2 {\rm d}s\right]<\infty$ by the fact that $(X_t)_{t\ge 0}$ is a solution of~\eqref{eq:sde}. The above observations are summarized in the following result.

\begin{lemma}\label{Lem:diff-eq-phi}
Let $(X_t)_{t\ge 0}$ be a solution of the SDE~\eqref{eq:sde} for some admissible initial condition and $u\in L^\infty(\R_{\ge 0},\R^d)$. Further let $\phi\in C^2(\R^d,\R)$ be such that 
$\nabla \phi$ is globally Lipschitz continuous. Then $\mathbb{E}[\phi(X_t)]_{t\ge 0}$ satisfies the differential equation~\eqref{eq:time_derivative_obs}.
\end{lemma}


\section{Feasibility of funnel control for Langevin dynamics}
\label{sec:feasibility}

In this section, we propose a modified funnel controller in order to achieve the control objective formulated in Section~\ref{Ssec:ContrObj}. The funnel controller is typically model-free (cf.~\cite{BergIlch21}) and only requires the information about the relative degree of the considered system to state the appropriate control law. Roughly speaking, the relative degree is the number of derivatives of the output which must be taken to obtain an explicit dependence on the input. For a precise definition for nonlinear ODE systems we refer to~\cite{Isid95}, for systems with infinite-dimensional internal dynamics see~\cite{BergPuch20a}. However, for controlled stochastic differential equations a concept of relative degree is not available. Nevertheless, for the output in~\eqref{eq:output}, it is possible to derive a relationship between~$\dot y$ and~$u$ by using~\eqref{eq:time_derivative_obs}, which gives for $\phi(x) = x_i$, $i=1,\ldots,d$, that
\begin{align*}
\ddt{}{}y_i(t) &= \mathbb{E}[\mathcal{L}^{u(t)} (X_t)_i] = \mathbb{E}\left[-\big(\nabla V(X_t) \!+\! A (X_t \!-\! u(t))\big) e_i\right] \\
&= -\mathbb{E}\left[\pd{V}{x_i}(X_t) \right] - e_i^\top A y(t) + e_i^\top A u(t),
\end{align*}
where $e_i$ is the $i$-th unit vector in $\R^d$, and therefore
\begin{equation}\label{eq:y-dot}
    \dot y(t) = -\mathbb{E}[\nabla V(X_t)] - Ay(t) + A u(t).
\end{equation}
This suggests that the SDE~\eqref{eq:sde} with output~\eqref{eq:output} at least exhibits an input-output behavior similar to that of a relative degree one system. This justifies to investigate the application of a corresponding funnel controller, which we need to modify here as follows
\begin{equation}\label{eq:fun-con}
u(t)=-\alpha \tanh\left(\frac{1}{\psi(t)-\|e(t)\|}\right) e(t),\ \ e(t)= y(t) - y_{\rm ref}(t)
\end{equation}
where $y_{\rm ref}\in W^{1,\infty}(\R_{\ge 0};\R^d)$ is the reference signal and $\alpha>0$ and $\psi\in\Psi$ are controller design parameters. The intuition behind the controller design~\eqref{eq:fun-con} is that the term $1/(\psi(t)-\|e(t)\|)$ is large whenever $\|e(t)\|$ is close to the funnel boundary, inducing a large control action. From the properties of the dynamics~\eqref{eq:y-dot} (more precisely, a high-gain property, cf.~\cite{BergIlch21}) it then follows that a large control action leads to a decaying tracking error or, in other words, the funnel boundary is repulsive.

The changes compared to a standard funnel controller as e.g.\ in~\cite{BergIlch21} are necessary to guarantee feasibility. Furthermore, depending on the constants from assumptions~(A1)--(A3), the controller will only be feasible for certain reference signals and a certain range of design parameters. For these signals and parameters we seek to show that, whenever $\|e(0)\| = \|\mathbb{E}[Z] - y_{\rm ref}(0)\| < \psi(0)$ for an admissible initial condition $X_0 = Z$, then there exists a unique solution of~\eqref{eq:sde},~\eqref{eq:output} under the control~\eqref{eq:fun-con} such that the tracking error evolves uniformly within the funnel $\cF_\psi$, i.e., $\|e(t)\| < \psi(t)$ for all $t\ge 0$. By a solution of~\eqref{eq:sde},~\eqref{eq:output},~\eqref{eq:fun-con} we mean a solution of the time-varying nonlinear SDE
\begin{equation}
\label{eq:sde-closed-loop}
\begin{aligned}
{\rm d} X_t &= -\Big(\nabla V(X_t) +  A\Big(X_t \\
&\left.\left. + \alpha \tanh\left(\tfrac{1}{\psi(t)-\|\mathbb{E}[X_t] - y_{\rm ref}(t)\|}\right) (\mathbb{E}[X_t] - y_{\rm ref}(t))\right) \right)\,{\rm d}t\\
&\, + \sqrt{2}\,{\rm d}B_t.
\end{aligned}
\end{equation}
In the following we present the main result of this paper.

\begin{theorem}\label{Thm:main}
Consider an SDE~\eqref{eq:sde} which satisfies assumptions~(A1)--(A3) with constants $c_1,\ldots,c_4$. Let $y_{\rm ref}\in W^{1,\infty}(\R_{\ge 0};\R^d)$, $\alpha >0$ and $\psi\in\Psi$ be such that there exists $p\in (0,1)$ with
\begin{equation}\label{eq:cond-c1c3}
   c_3 a \alpha \|\psi\|_\infty =  p c_1,
\end{equation}
where $a := \|A\|$, and
\begin{equation}\label{eq:cond-c1234}
\begin{aligned}
     &\frac{1}{1-p} \left(c_4 + \frac{c_2 c_3}{c_1}\right) + \frac{ap}{1-p} \|\psi\|_\infty  + \frac{a}{1-p} \|y_{\rm ref}\|_\infty \\
     & + \|\dot y_{\rm ref}\|_\infty + \|\dot \psi\|_\infty < \frac{pqc_1}{2c_3},
\end{aligned}
\end{equation}
where $q:= \frac{\inf_{t\ge 0} \psi(t)}{\sup_{t\ge 0} \psi(t)}$. Furthermore, let $X_0 = Z$ be an admissible initial condition which satisfies
\begin{equation}\label{eq:ini-cond}
    \|\mathbb{E}[Z] - y_{\rm ref}(0)\| < \psi(0)
\end{equation}
and
\begin{equation}\label{eq:kappa}
\begin{aligned}
    &\mathbb{E}\big[V(Z) + \tfrac12 Z^\top A Z\big] \\
    &\le \frac{c_2 + c_4 a \alpha \|\psi\|_\infty  + a^2 \alpha \|\psi\|_\infty (\|\psi\|_\infty  + \|y_{\rm ref}\|_\infty)}{c_1 - c_3 a \alpha \|\psi\|_\infty} =: \kappa.
\end{aligned}
\end{equation}
Then the SDE~\eqref{eq:sde} with output~\eqref{eq:output} and under the control~\eqref{eq:fun-con} (i.e., the SDE~\eqref{eq:sde-closed-loop}) has a unique solution $(X_t)_{t\ge 0}$ which satisfies
\begin{equation}\label{eq:error-funnel}
    \exists\, \eps>0\ \forall\, t\ge 0:\ \|\mathbb{E}[X_t] - y_{\rm ref}(t)\| < \psi(t) - \eps.
\end{equation}
\end{theorem}
\begin{proof} \emph{Step 1}: We show the existence of a unique solution of~\eqref{eq:sde-closed-loop}. To this end, define
\begin{align*}
    &d:\R_{\ge 0}\times\R^d\to\R^d,\\
     &(t,z) \mapsto \begin{cases}  \alpha \tanh\left(\frac{1}{\psi(t)-\|z - y_{\rm ref}(t)\|}\right) (z - y_{\rm ref}(t)), \\
      \qquad\qquad\qquad\qquad\qquad\mbox{if}\ \|z - y_{\rm ref}(t)\|< \psi(t),\\
    \alpha \psi(t) \frac{z - y_{\rm ref}(t)}{\|z - y_{\rm ref}(t)\|},\\
     \qquad\qquad\qquad\qquad\qquad\mbox{if}\ \|z - y_{\rm ref}(t)\|\ge \psi(t).\end{cases}
\end{align*}
It is easy to see that $d\in L^\infty(\R_{\ge 0}\times\R^d,\R^d)$ since $\psi$ is bounded. Therefore, by Lemma~\ref{Lem:SDE-mod-sln} there exists a unique solution $(X_t)_{t\ge 0}$ of the SDE~\eqref{eq:sde-mod} with initial condition $X_0 = Z$. Define
\[
    u(t) := -d(t,\mathbb{E}[X_t]),\quad t\ge 0,
\]
then, by construction of~$d$, it is clear that, if~\eqref{eq:error-funnel} holds, then $(X_t)_{t\ge 0}$ is also the unique solution of~\eqref{eq:sde-closed-loop} and~$u$ coincides with the control signal in~\eqref{eq:fun-con}.

\emph{Step 2}: We show~\eqref{eq:error-funnel}. To this end, we define
\[
    \tilde{V}(x) := V(x) + \frac{1}{2}x^\top A x
\]
and, which is the key idea of the proof, consider the observable
\[
    z(t) := \mathbb{E}[\tilde  V(X_t)], \quad t\ge 0.
\]

\emph{Step 2a}: We derive an estimate for the derivative of $z$. Since $V$ is non-negative by~(A1) we have $z(t) \ge 0$ for all $t\ge 0$, and by Lemma~\ref{Lem:diff-eq-phi} with $\phi = \tilde V$ we obtain (since $\nabla \tilde V$ is globally Lipschitz continuous by~(A1))
\begin{align*}
  \dot z(t) &= \mathbb{E}\left[-\big(\nabla V(X_t) \!+\! A (X_t \!-\! u(t))\big)^\top \nabla \tilde{V}(X_t) \!+\! \Delta \tilde{V}(X_t) \right]\\
  &= \mathbb{E}[\mathcal{L}^0 \tilde  V(X_t)] + \mathbb{E}[\nabla  V(X_t)]^\top A u(t) + y(t)^\top A^2 u(t),
\end{align*}
where $\mathcal{L}^0$ is the operator~\eqref{eq:def_generator} for $v = 0$. Using condition~(A2) we find that
\begin{align*}
     \mathbb{E}[\mathcal{L}^0 \tilde  V(X_t)] &= \mathbb{E}[\Delta V(X_t) + \tr A - \|\nabla V(X_t) + A X_t\|^2] \\
    &\le \mathbb{E}[-c_1 \tilde V(X_t) + c_2] = - c_1 z(t) + c_2.
\end{align*}
Under condition~(A3) we obtain
\[
    \|\mathbb{E}[\nabla V(X_t)]\| \le \mathbb{E}[c_3 \tilde V(X_t) + c_4] \le c_3 z(t) + c_4,
\]
thus
\[
    \dot z(t) \le - c_1 z(t) + c_2 + (c_3 z(t) + c_4)  \|A u(t)\| + y(t)^\top A^2 u(t).
\]
Now, let
\[
    T:= \inf \setdef{t\ge 0}{ \|y(t) - y_{\rm ref}(t)\| = \psi(t)}
\]
and by~\eqref{eq:ini-cond} we have that $T\in (0,\infty]$ and
\[
    \|y(t)\| \le \|y(t) - y_{\rm ref}(t)\| + \|y_{\rm ref}(t)\| \le \psi(t) + \|y_{\rm ref}(t)\|
\]
for all $t\in [0, T)$. By definition of~$u(t)$ and~$d$ we find that
\[
    \|u(t)\| \le \alpha \psi(t),\quad t\ge 0,
\]
and with $a = \|A\|$ and $\|\psi\|_\infty = \sup_{t\ge 0} \psi(t)$ we obtain
\begin{align*}
    \dot z(t) &\le - c_1 z(t) + c_2 + a \alpha \|\psi\|_\infty (c_3 z(t) + c_4)  \\
    &\quad + a^2 \alpha \|\psi\|_\infty (\|\psi\|_\infty + \|y_{\rm ref}\|_\infty)
\end{align*}
for all $t\in [0,T)$. Then it follows from the comparison principle and~\eqref{eq:kappa} that
\[
   \forall\, t \in [0,T):\ z(t) \le \kappa.
\]

\emph{Step 2b}: We define a suitable $\eps>0$ for~\eqref{eq:error-funnel}. By assumption~\eqref{eq:cond-c1234} there exists $p\in(0,1)$ such that
\begin{align*}
   &\left(c_3 \frac{c_2 + c_4 M  + a M  (\|\psi\|_\infty  + \|y_{\rm ref}\|_\infty)}{c_1 - c_3 M} + c_4\right)  + a \|y_{\rm ref}\|_\infty
    \\
    &+ \|\dot y_{\rm ref}\|_\infty + \|\dot \psi\|_\infty < \frac{q}{2} M
\end{align*}
is satisfied for $M := p \frac{c_1}{c_3}$. Since by~\eqref{eq:cond-c1c3} we have that additionally $M = a \alpha \|\psi\|_\infty$ it follows that
\begin{align*}
    &\left(c_3 \frac{c_2 \!+\! c_4 a \alpha \|\psi\|_\infty  \!+\! a^2 \alpha \|\psi\|_\infty (\|\psi\|_\infty  \!+\! \|y_{\rm ref}\|_\infty)}{c_1 \!-\! c_3 a \alpha \|\psi\|_\infty} \!+\! c_4\right)  \\
    &\quad + a \|y_{\rm ref}\|_\infty
    + \|\dot y_{\rm ref}\|_\infty + \|\dot \psi\|_\infty < \alpha a \lambda/2,
\end{align*}
where $\lambda:= \inf_{t\ge 0} \psi(t)$. Therefore, with $\kappa$ from~\eqref{eq:kappa} it follows that
\[
     c_3 \kappa + c_4 + a \|y_{\rm ref}\|_\infty
    + \|\dot y_{\rm ref}\|_\infty + \|\dot \psi\|_\infty < \alpha a \lambda/2,
\]
and hence there exists $\eps >0$ such that $\eps < \min\{\psi(0)-\|e(0)\|, \lambda/2\}$ and
\[
   c_3 \kappa + c_4 + a \|y_{\rm ref}\|_\infty
    + \|\dot y_{\rm ref}\|_\infty - a \alpha \tanh(1/\eps) \lambda/2\le -\|\dot \psi\|_\infty .
\]

\emph{Step 2c}: We show that  the tracking error $e(t) = y(t) - y_{\rm ref}(t)$ satisfies $\|e(t)\|\le \psi(t) - \eps$ for all $t \in [0,T)$. Seeking a contradiction, and invoking $\|e(0)\| < \psi(0) - \eps$, assume that $\|e(t_1)\|> \psi(t_1) - \eps$ for some $t_1\in [0,T)$ and define
\[
    t_0 := \max \setdef{t\in [0,t_1]}{\|e(t)\| = \psi(t) - \eps}.
\]
Then we have $\psi(t)-\|e(t)\| \le \eps$ and since $\eps < \lambda/2$, we have $\|e(t)\| \ge \psi(t) - \eps > \lambda/2$ for $t\in [t_0, t_1]$. Therefore, $\|u(t)\|  \ge \alpha \tanh(1/\eps) \lambda/2$ for all $t\in [t_0, t_1]$ and by~\eqref{eq:y-dot} we obtain
\begin{align*}
    &\tfrac12 \ddt \|e(t)\|^2 = -e(t)^\top \mathbb{E}[\nabla V(X_t)]  - e(t)^\top Ay(t)
    \\
    &\quad + e(t)^\top A u(t) - e(t)^\top \dot y_{\rm ref}(t)\\
&\le \|e(t)\|\big( c_3 z(t) + c_4  - a \|e(t)\| + a \|y_{\rm ref}(t)\|
- a \|u(t)\| \\
&\quad + \|\dot y_{\rm ref}(t)\|\big)\\
&\le \|e(t)\|\big( c_3 \kappa + c_4 + a \|y_{\rm ref}\|_\infty
    - a \alpha \tanh(1/\eps) \lambda/2 \\
    &\quad + \|\dot y_{\rm ref}\|_\infty\big)\\
    &\le - L \|e(t)\|,
\end{align*}
where $L := \|\dot \psi\|_\infty$ and by the mean value theorem we have
\[
    |\psi(t_1)-\psi(t_0)| \le L |t_1 - t_0|.
\]
Upon integration we obtain
\begin{align*}
  &\|e(t_1)\| - \|e(t_0)\| = \int_{t_0}^{t_1} \tfrac12 \|e(t)\|^{-1} \ddt  \|e(t)\|^2 {\rm d}t\\
  &\le -L (t_1 - t_0) \le - |\psi(t_1) - \psi(t_0)| \le \psi(t_1) - \psi(t_0),
\end{align*}
thus arriving at the contradiction
\[
    \eps = \psi(t_0) - \|e(t_0)\| \le \psi(t_1) - \|e(t_1)\| < \eps.
\]
In particular, this implies $T=\infty$ and  we have further shown~\eqref{eq:error-funnel} and this concludes the proof.
\end{proof}

\begin{remark}\label{rem:simpl_setting}
Some comments on the conditions in Theorem~\ref{Thm:main} are warranted. First observe that, due to assumptions~(A2) and~(A3), the constants $c_1,\ldots,c_4$ depend on~$a = \|A\|$ and will increase/decrease when $a$ changes, see Section~\ref{sec:examples} for a specific example.

In order to check the conditions~\eqref{eq:cond-c1c3} and~\eqref{eq:cond-c1234}, suitable values for the design parameters must be found. To this end, note that the controller weighting matrix $A$ is also a design parameter, which may be chosen as desired in order to satisfy the assumptions. A typical situation is that the right-hand side in~\eqref{eq:cond-c1234} grows faster with increasing~$a$ than the left-hand side, see e.g. Section~\ref{sec:examples}. Then arbitrary $y_{\rm ref}$, $\psi$ and $p$ may be fixed and afterwards $a$ can be chosen sufficiently large so that~\eqref{eq:cond-c1234} is satisfied. After that, $\alpha$ can be defined so that~\eqref{eq:cond-c1c3} is satisfied~-- note that~\eqref{eq:cond-c1234} is independent of~$\alpha$.

The conditions in Theorem~\ref{Thm:main} simplify if we choose a constant funnel boundary. For $\psi = {\rm const}$, we have that $\dot \psi=0$ and $q=1$ in~\eqref{eq:cond-c1234}. Choosing $p=\tfrac12$, condition~\eqref{eq:cond-c1234} turns into
\begin{equation}\label{eq:cond-c1234-simpl}
     2 \left(c_4 + \frac{c_2 c_3}{c_1}\right) + a \big(\psi + 2 \|y_{\rm ref}\|_\infty\big) + \|\dot y_{\rm ref}\|_\infty < \frac{c_1}{4c_3},
\end{equation}
which, for fixed $\psi$, only involves the constants $a,c_1,\ldots, c_4$ and the reference signal $y_{\rm ref}$.
\end{remark}

\begin{remark}
The assumptions of Theorem~\ref{Thm:main} are structural in the following sense: For fixed controller design parameters $\alpha>0$, $\psi\in\Psi$ and symmetric positive definite $A\in\R^{d\times d}$, and reference signal $y_{\rm ref}\in W^{1,\infty}(\R_{\ge 0},\R^d)$, there is a whole class of systems which satisfy the assumptions of Theorem~\ref{Thm:main}, if they are satisfied for at least one potential $V$. More precisely, for $a = \|A\|$ the set
\[
    \Sigma_{a,\alpha,\psi} \!=\! \setdef{\!V\in C^2(\R^d,\R)}{ \!\!\!\begin{array}{l} V\ \text{satisfies (A1)--(A3) and} \\ \text{\eqref{eq:cond-c1c3},~\eqref{eq:cond-c1234} for some $p\in(0,1)$}\end{array}\!\!\!}
\]
is either empty or contains an open ball in $C^2(\R^d,\R)$, as all conditions depend continuously on $V$ and its first two derivatives.

We stress that indeed $\Sigma_{a,\alpha,\psi}$ is not always empty. It contains the potential $V=0$ under the condition
\begin{equation}\label{eq:cond-V=0}
\begin{aligned}
    &\frac{2pd}{(1-p) \alpha \|\psi\|_\infty } + \frac{ap}{1-p} \|\psi\|_\infty  + \frac{a}{1-p} \|y_{\rm ref}\|_\infty \\
     &\quad+ \|\dot y_{\rm ref}\|_\infty + \|\dot \psi\|_\infty < \frac{q}{4} a \alpha \|\psi\|_\infty
\end{aligned}
\end{equation}
on the controller design parameters and $y_{\rm ref}$, where $q= \frac{\inf_{t\ge 0} \psi(t)}{\sup_{t\ge 0} \psi(t)}$. This condition results from the fact that, by Example~\ref{Ex:quad-pot}, (A1)--(A3) are satisfied for $c_1 = a$, $c_2=da$, $c_4=0$ and arbitrary $c_3\ge 0$ in this case. Furthermore, for $p\in(0,1)$ we find that~\eqref{eq:cond-c1c3} always holds for $c_3 = 2p/\alpha \|\psi\|_\infty$, so we fix $c_3$ to this value. Then inserting this into~\eqref{eq:cond-c1234} leads to the condition~\eqref{eq:cond-V=0}.
\end{remark}

\section{Numerical Example}
\label{sec:examples}

In this section, we show that funnel control can be used for tracking control of a stochastic system with a more complex energy function $V$ than previously discussed. We also illustrate the fact that funnel control is essentially model-free, i.e., for a fixed tuple of controller design parameters $(\psi, A, \alpha)$ it is feasible for a whole class of systems that satisfy the assumptions of Theorem~\ref{Thm:main}.

We consider diffusion in the two-dimensional double-well potential
\[ V_{\rm dw}(x, y) = C_x (x^2 - 1)^2 + C_y y^2, \]
where the second parameter is set to $C_y = 3$, while we will establish a corresponding range of admissible values for $C_x$ further below. The double-well is a very simple, but widely used model system for molecular applications that involve metastability. A contour plot of the potential for $C_x = 1, C_y = 3$ is shown in Figure~\ref{fig:dw_potential}~A. In the uncontrolled setting ($u \equiv 0$), the dynamics spend long times oscillating around one of the two potential minima, while rarely crossing the barrier at $x = 0$. Therefore, we choose a reference signal which ensures that the controlled system alternates frequently between the two minima, following a figure-eight shaped trajectory (also shown in Figure~\ref{fig:dw_potential}~A), given by:
\begin{align*}
y_{\rm ref}(t) &= \begin{pmatrix} \cos(\frac{2\pi}{\rho}t) & \sin(\frac{4\pi}{\rho}t)  \end{pmatrix}^\top,\\
 \dot{y}_{\rm ref}(t) &= \frac{4\pi}{\rho}\begin{pmatrix} -\frac{1}{2}\sin(\frac{2\pi}{\rho}t) & \cos(\frac{4\pi}{\rho}t)  \end{pmatrix}^\top.
\end{align*}
The period $\rho$ of the reference signal is set to $0.5$, while the simulation horizon is $T = 1.0$, thus enforcing two complete oscillations along the reference trajectory. We verify numerically that $\|y_{\rm ref}\|_\infty \approx 1.25$ and $\|\dot{y}_{\rm ref}\|_\infty \approx 28.1$.

In virtue of Remark~\ref{rem:simpl_setting}, we consider the simple setting of a constant funnel boundary  $\psi = 1.0$, by which $q = 1$ in~\eqref{eq:cond-c1234}. Furthermore, we choose the controller weight matrix $A = a I_2$ with control strength $a > 0$ and fix $p=\tfrac12$, by which we may consider the simplified version~\eqref{eq:cond-c1234-simpl} of the aforementioned condition.

To ensure that $\nabla V$ is actually globally Lipschitz continuous, in accordance with our theoretical results, we will fix some $R > 1$, and modify the potential to be quadratic outside a ball of radius $R$:
\[
    V:\R^2\to\R,\ (x,y)\mapsto \begin{cases} V_{\rm dw}(x, y), & |x| \le R,\\
   d_1 x^2 \!-\! d_2 x \!+\! d_3 \!+\! C_y y^2, & x > R,\\
    d_1 x^2 \!+\! d_2 x \!+\! d_3 \!+\! C_y y^2, & x < -R. \end{cases}
\]
The constants $d_1,d_2,d_3\in\R$ are uniquely determined by the condition that $V\in C^2(\R^2,\R)$ and can be calculated to be
\[
    d_1 = 2 C_x (3R^2-1),\quad d_2 = 8C_x R^3,\quad d_3 = C_x (3R^4+1).
\]
We will now show that all assumptions of Theorem~\ref{Thm:main} are satisfied for this example.

\begin{figure}
\includegraphics[width=.43\textwidth]{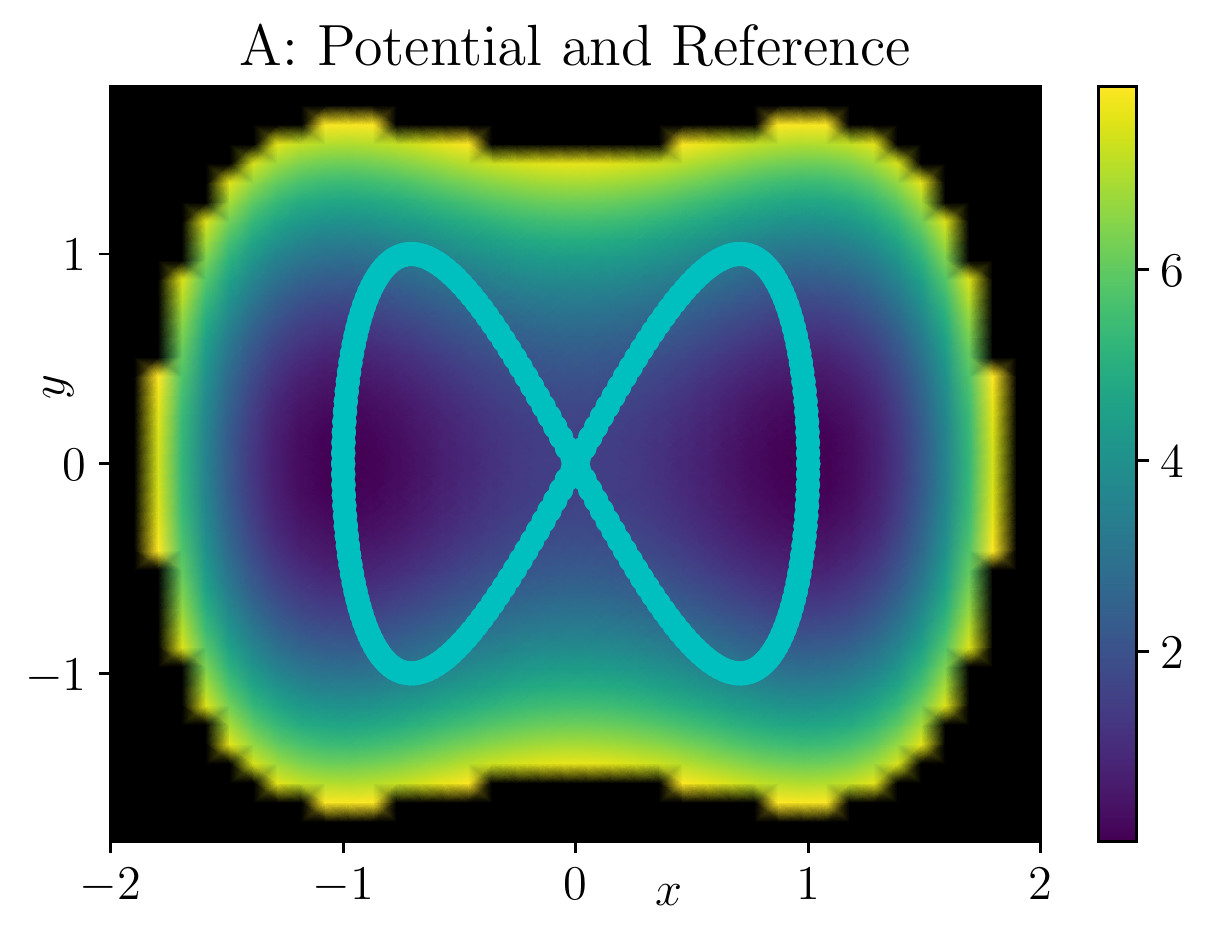}
\includegraphics[width=.45\textwidth]{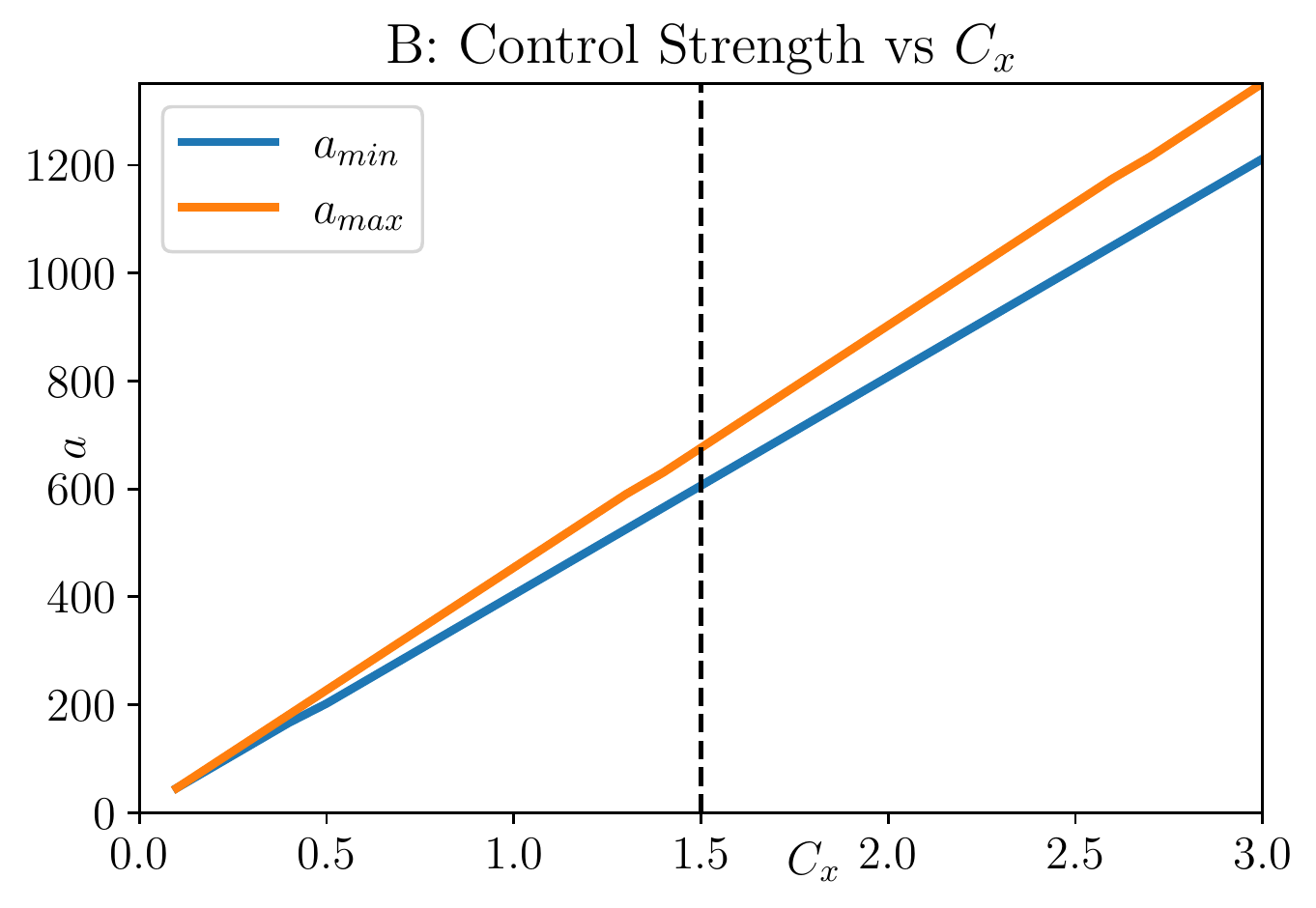}
\caption{A: Contour plot of the double-well potential for $C_x = 1.5$ and $C_y = 3.0$, with the reference signal $y_{\rm ref}$ shown in light blue. B: Lower and upper bounds for the control strength~$a$ as a function of $C_x$. The value $C_x = 1.5$ used in numerical simulations is indicated by the dashed line. \label{fig:dw_potential}}
\end{figure}

\subsection{Theoretical Performance Guarantees}
In the following, we calculate the constants $c_1,\ldots,c_4$ from assumptions~(A2) and~(A3) and determine the minimal $a$ such that~\eqref{eq:cond-c1234-simpl} holds. Finally, we determine $\alpha$ using~\eqref{eq:cond-c1c3}. By definition, $V$ satisfies assumption~(A1). For the remaining conditions, we require  the following derivatives:
\begin{align*}
(\nabla V)_1(x,y) &= \begin{cases} 4C_x x (x^2-1), & |x| \le R,\\
   2d_1 x - d_2 , & x > R,\\
   2 d_1 x + d_2, & x < -R, \end{cases}\\
   (\nabla V)_2(x,y) &= 2C_y y, \\
\Delta V(x,y) &= \begin{cases} 12 C_x x^2 - 4 C_x + 2 C_y, & |x| \le R,\\
  2 (d_1 + C_y) , & |x| > R, \end{cases} &
\end{align*}

\paragraph{Case 1}
We first consider the case of $(x,y)\in\R^2$ with $|x| \le R$. Concerning condition (A2), we verify that
\begin{align*}
  & \Delta V(x,y) + \tr A - \|\nabla V(x,y) + A\begin{smallpmatrix} x\\ y\end{smallpmatrix} \|^2 \\
  &= 12 C_x x^2 \!-\! 4 C_x \!+\! 2 C_y \!+\! 2a \!-\! \big(4 C_x (x^2\!-\!1) \!+\! a\big)^2 (x^2\!-\!1)\\
   &\quad -\big(4 C_x (x^2\!-\!1) \!+\! a\big)^2 \!-\! (2C_y \!+\! a)^2 y^2\\
  &= 12 C_x x^2 - 4 C_x + 2 C_y + 2a - 16 C_x^2 (x^2-1)^2 (x^2-1) \\
  &\quad - 8 a C_x (x^2-1)^2 - a^2 (x^2-1) - 16 C_x^2 (x^2-1)^2\\
  &\quad    - 8 a C_x (x^2-1) - a^2 - (2C_y + a)^2 y^2\\
  &\le 12 C_x x^2 - 4 C_x + 2 C_y + 2a + 8 a C_x - 8 a C_x (x^2-1)^2 \\
  &\quad - (a^2+ 8 a C_x) x^2 - (2C_y + a)^2 y^2 \\
   & = - 8 a C_x (x^2-1)^2  \!-\! (a^2 + 8 a C_x - 12 C_x) x^2 \!-\! (2C_y + a)^2 y^2 \\
   &\quad - 4 C_x + 2 C_y + 2a + 8 a C_x   \\
   &\le -c_1\big( C_x(x^2-1)^2 + \tfrac{a}{2} x^2 +  \tfrac{1}{2}(2C_y +a) y^2\big) + c_2 \\
   &= -c_1 \big(V(x,y) + \tfrac12 \begin{smallpmatrix} x\\ y\end{smallpmatrix}^\top A \begin{smallpmatrix} x\\ y\end{smallpmatrix}\big) + c_2
\end{align*}
will hold for
\begin{equation}
\label{eq:c1_c2_dw_leq_R}
\begin{aligned}
  c_1 &:= \min\left\{8 a, 2a + 16 C_x - 24 \tfrac{C_x}{a},  4C_y + 2a\right\},\\
  c_2 &:= 2 C_y + 2a + 8 a C_x - 4 C_x.
\end{aligned}
\end{equation}
Concerning condition (A3), we find that
\begin{align*}
  & \|\nabla V(x,y)\| = \sqrt{16 C_x^2 x^2 (x^2-1)^2 + 4 C_y^2 y^2} \\
  &\le 4 C_x \, |x|\, |x^2-1| + 2 C_y |y|\\
  &= 4 C_x \, |x|\, |x^2-1| + 2 C_y |y| - c_3 \big( C_x (x^2-1)^2 + C_y y^2 \\
  &\quad + \tfrac{a}{2} (x^2+y^2)\big) - c_4\\
  &\quad + c_3 \big( C_x (x^2-1)^2 + C_y y^2 + \tfrac{a}{2} (x^2+y^2)\big) + c_4\\
    &= -c_3 C_x \left( \big( |x^2\!-\!1| \!-\! \tfrac{2|x|}{c_3}\big)^2 +\big( \tfrac{a}{2 C_x} \!-\! \tfrac{4}{c_3^2}\big) x^2\right) \\
     &\quad - \tfrac{c_3}{2} (2 C_y \!+\!a) \left(|y| \!-\! \tfrac{2 C_y}{c_3(2C_y \!+\!a)}\right)^2 \!+\! \tfrac{2 C_y^2}{c_3(2C_y \!+\!a)} \!-\! c_4 \\
  &\quad  + c_3 \big( C_x (x^2-1)^2 + C_y y^2 + \tfrac{a}{2} (x^2+y^2)\big) + c_4\\
  &\le c_3 \big( C_x (x^2-1)^2 + C_y y^2 + \tfrac{a}{2} (x^2+y^2)\big) + c_4
\end{align*}
will be satisfied for the choice
\begin{align}
\label{eq:c3_c4_dw_leq_R}
  c_3 &:= \sqrt{\frac{8C_x}{a}}, &
  c_4 &:= \sqrt{\frac{a}{8 C_x}} \frac{2 C_y^2}{2C_y +a} .
\end{align}
For
\begin{equation}\label{eq:cond-Cx-Cy}
    C_y \le \min\left\{ \frac{3a}{2}, \frac{4a-6}{a} C_x\right\}
\end{equation}
we observe that $c_1$ simplifies to $c_1 = 4C_y + 2a$.

\paragraph{Case 2} We now turn to the case $x > R$ and show that for $a$ and $R$ sufficiently large the assumptions~(A2) and~(A3) are satisfied with the same constants $c_1,\ldots,c_4$. For~(A2) we observe that
\begin{align*}
  & \Delta V(x,y) + \tr A - \|\nabla V(x,y) + A\begin{smallpmatrix} x\\ y\end{smallpmatrix} \|^2 \\
  &= 2d_1 + 2C_y + 2a - ((2d_1+a)x-d_2)^2 - (2C_y + a)^2 y^2\\
  &= 2(d_1 + C_y + a) - d_2^2 - (2d_1+a)^2x^2 + 2 (2d_1+a) d_2 x \\
  &\quad- (2C_y + a)^2 y^2 \\
   &\le -c_1\big( \tfrac{1}{2}(2 d_1 +a)x^2 -d_2 x + d_3 +  \tfrac{1}{2}(2C_y +a) y^2\big) + c_2 \\
   &= -c_1 \big(V(x,y) + \tfrac12 \begin{smallpmatrix} x\\ y\end{smallpmatrix}^\top A \begin{smallpmatrix} x\\ y\end{smallpmatrix}\big) + c_2
\end{align*}
holds with $c_1 = 4C_y + 2a$, if
\begin{align*}
 & - (2d_1+a)^2x^2 + 2 (2d_1+a) d_2 x \\
 &\le -(2C_y +a)\big( (2 d_1 +a)x^2 - 2d_2 x\big)\\
 \text{and}\quad &2(d_1 + C_y + a) - d_2^2 \le  -(2C_y +a) d_3 + c_2.
\end{align*}
After inserting~$c_2$, the second condition is equivalent to
\begin{align}
  &(8a-4) C_x \ge  2d_1 - d_2^2 + (2C_y +a)d_3 \notag\\
  \iff\quad & 64 C_x R^6 - (2C_y+a)(3R^4+1) - 12 R^2 + 8 a \ge 0\notag\\
  \iff\quad & (3R^4-7)^{-1} \big(64 C_x R^6 \!-\! 6C_y R^4 \!-\! 12 R^2 \!-\! 2C_y\big) \ge a,  \label{eq:upper-bound-a-R-2}
\end{align}
which is valid for $R>1$ such that $3R^4 > 7$ (which we suppose henceforth) and defines an upper bound for $a$. Invoking $x>R$ the first condition simplifies to
\[
   \forall\, x>R:\quad 2 d_2 (d_1 - C_y)  \le (2 d_1 +a)x  (d_1 - C_y).
\]
We choose $R$ sufficiently large so that
\begin{equation}\label{eq:R-suff-1}
    d_1 = 2 C_x (3R^2-1) \ge C_y,
\end{equation}
then the above condition is satisfied, if
\[
    2 d_2  \le (2 d_1 +a) R,
\]
which is satisfied for $a$ sufficiently large, more precisely for
\begin{equation}\label{eq:lower-bound-a-R}
    a \ge \frac{2d_2}{R} - 2 d_1 = 16C_x R^2 - 4 C_x (3R^2-1) = 4 C_x (R^2+1).
\end{equation}
This lower bound is compatible with the upper bound~\eqref{eq:upper-bound-a-R-2}, if
\begin{equation}\label{eq:R-suff-4}
    52 C_x R^6 - 6(2C_x+C_y)R^4 + (28C_x-12) R^2 + 28C_x - 2 C_y \ge 0.
\end{equation}
For~(A3)  we first observe that
\[
    2 d_1 R - d_2 = 4 C_x R (R^2-1) \ge 0
\]
for $R> 1$. Then we obtain for $x>R$ that
\begin{align*}
   &\|\nabla V(x,y)\| = \sqrt{(2d_1x-d_2)^2 + 4 C_y^2 y^2}\\
   & \le 2 d_1 x - d_2 + 2 C_y |y|\\
    &= -\frac{c_3}{2} (2d_1 +a) \left(x \!-\! \frac{2d_1 + c_3 d_2}{c_3(2d_1+a)}\right)^2 + \frac{(2d_1 + c_3 d_2)^2}{2c_3(2d_1+a)} \!-\! d_2 \\
    &\quad - \tfrac{c_3}{2} (2 C_y +a) \left(|y| - \tfrac{2 C_y}{c_3(2C_y +a)}\right)^2+ \tfrac{2 C_y^2}{c_3(2C_y +a)} \\
  &\quad  + c_3 \big( d_1 x^2 - d_2 x + C_y y^2 + \tfrac{a}{2} (x^2+y^2)\big)\\
  &\le c_3 \big( d_1 x^2 - d_2 x + d_3 + C_y y^2 + \tfrac{a}{2} (x^2+y^2)\big) + c_4
\end{align*}
holds, if
\[
     \frac{(2d_1 + c_3 d_2)^2}{2c_3(2d_1+a)} - d_2 \le c_3 d_3.
\]
Invoking $\frac{1}{2d_1+a} \le \frac{R}{2d_2}$ this is true, if
\[
    \left(\frac{d_1^2}{c_3 d_2} + d_1 + \frac{c_3d_2}{4}\right) R - d_2 \le c_3 d_3.
\]
This leads to
\begin{align*}
    \frac{R d_1^2}{d_2}  &\le c_3^2 \left( d_3- \frac{Rd_2}{4}\right) + 2 c_3 C_x R (R^2+1) \\
    &= \frac{8C_x^2}{a} (R^4+1) + 2 \sqrt{\frac{8C_x}{a}} C_x R (R^2+1),
\end{align*}
hence
\begin{equation}\label{eq:upper-bound-a-R}
    a - 4 \sqrt{8C_x}   \frac{R^3 (R^2+1)}{(3R^2-1)^2} \sqrt{a} - 16C_x \frac{R^2  (R^4+1)}{(3R^2-1)^2} \le 0.
\end{equation}
The above condition defines a second upper bound for $a$ in terms of $R$. We need to ensure that this upper bound actually exceeds the lower bound given by~\eqref{eq:lower-bound-a-R}, so that feasible values of $a$ exist. We first observe that for $R$ large enough the argument of the minimum of the left hand side of~\eqref{eq:upper-bound-a-R} in~$a$ is less than the lower bound from~\eqref{eq:lower-bound-a-R}:
\begin{equation}\label{eq:R-suff-2}
\begin{aligned}
    &32 C_x   \frac{R^6 (R^2+1)^2}{(3R^2-1)^4} \le 4 C_x (R^2+1)\\
     \iff\quad &8 R^6 (R^2+1) \le(3R^2-1)^4.
\end{aligned}
\end{equation}
Then both bounds are compatible, if for $a=4 C_x (R^2+1)$ the left hand side of~\eqref{eq:upper-bound-a-R} is negative, which is the case if, and only if,
\begin{multline}\label{eq:R-suff-3}
   4 (R^2+1) (3R^2-1)^2 - 8 \sqrt{8} R^3 (R^2+1) \sqrt{R^2+1} \\
   - 16 R^2 (R^4+1) < 0,
\end{multline}
and indeed this is true for $R$ large enough. Summarizing, if $R> 1$ is large enough so that~\eqref{eq:R-suff-1},~\eqref{eq:R-suff-4},~\eqref{eq:R-suff-2},~\eqref{eq:R-suff-3} hold, then there exists an interval $\mathcal{A}_1 \subset \R^+$ such that~\eqref{eq:upper-bound-a-R-2},~\eqref{eq:lower-bound-a-R} and~\eqref{eq:upper-bound-a-R} hold for all $a \in \mathcal{A}_1$. The considerations for the case $x<-R$ are analogous and omitted. Finally, using the expressions~\eqref{eq:c1_c2_dw_leq_R} and~\eqref{eq:c3_c4_dw_leq_R} for the constants $c_1, c_2, c_3, c_4$, the condition~\eqref{eq:cond-c1234-simpl} reads
\begin{equation}\label{eq:condition_c1234_dw}
\begin{aligned}
&\sqrt{\frac{a}{8 C_x}} \frac{4 C_y^2}{2C_y +a} + 2 \sqrt{\frac{8C_x}{a}} \frac{2 C_y + 2a + 8 a C_x - 4 C_x}{4C_y + 2a} \\
&+ a (\psi  + 2\|y_{\rm ref}\|_\infty) + \|\dot y_{\rm ref}\|_\infty < \sqrt{\frac{a}{8 C_x}} (C_y + \tfrac12 a).
\end{aligned}
\end{equation}
This condition leads to a refined interval $\mathcal{A}_2 = [a_{\min}, a_{\max}] \subset \mathcal{A}_1$ of admissible control parameters $a$. In Figure~\ref{fig:dw_potential} B, we show the upper and lower bounds $a_{\min}$ and $a_{\max}$ as a function of $C_x$. These considerations also show that for a given $a > 0$, there is an interval of model parameters $C_x$ such that the application of funnel control with control strength~$a$ is feasible for all $C_x$ in that interval.

\subsection{Numerical Results}

For the numerical validation of our results, we consider the specific setting $C_x = 1.5, C_y = 3.0$, which also satisfy~\eqref{eq:cond-Cx-Cy}, thus simpifying the first constant to $c_1 = 4C_y + 2a$. We verify that for $R = 10.0$, the conditions~\eqref{eq:R-suff-1},~\eqref{eq:R-suff-4},~\eqref{eq:R-suff-2},~\eqref{eq:R-suff-3} are satisfied, while noting that this value of $R$ is so large that it suffices to consider $V=V_{\rm dw}$ in the numerical simulations. The conditions~\eqref{eq:upper-bound-a-R-2},~\eqref{eq:lower-bound-a-R},~\eqref{eq:upper-bound-a-R}, and~\eqref{eq:condition_c1234_dw} lead to admissible values of $a$ in
\[
    \mathcal{A}_2 = [606, 676],
\]
based on which we choose the minimal control strength $a = 606$. Finally, we determine $\alpha$ for this choice of $a$ according to~\eqref{eq:cond-c1c3}, obtaining $\alpha \approx 7.18$.

We then apply the funnel controller~\eqref{eq:fun-con} to track the reference signal under the dynamics of~$20$ independent trajectories, simulated by the Euler-Maruyama scheme at elementary integration time step $10^{-4}$. At each time step, we calculate the outputs $y_1(t), y_2(t)$ required to compute the feedback control by averaging over these~$20$ trajectories.

With these settings, the funnel controller~\eqref{eq:fun-con} applied to the SDE~\eqref{eq:sde} with output~\eqref{eq:output} achieves an impressive tracking performance. We confirm in Figures~\ref{fig:dw_tracking}~A and~B that there is almost no difference between the prescribed mean values and the empirical means of the controlled trajectories. In fact, the norm $\|e(t)\|$ of the error vector remains significantly smaller than the funnel boundary throughout the simulation horizon, as shown in Figure~\ref{fig:dw_tracking}~D. The required control action $Au(t)$ is of the same order of magnitude as $a$, as shown in Figure~\ref{fig:dw_tracking}~C, which confirms an outstanding controller performance.

Lastly, we show that the provided interval for $\mathcal{A}_2$ based on theoretical guarantees is actually quite conservative. We repeat the above experiment with $a = 5.0$, while all other settings remain unchanged. The results are shown in Figure~\ref{fig:dw_tracking_min}. We find that the distance between the tracking error and the funnel boundary is now reduced, also resulting in a significantly larger, but still acceptable standard deviation. The control action $Au(t)$, on the other hand, is reduced by one to two orders of magnitude.

\begin{figure}
\includegraphics[width=.45\textwidth]{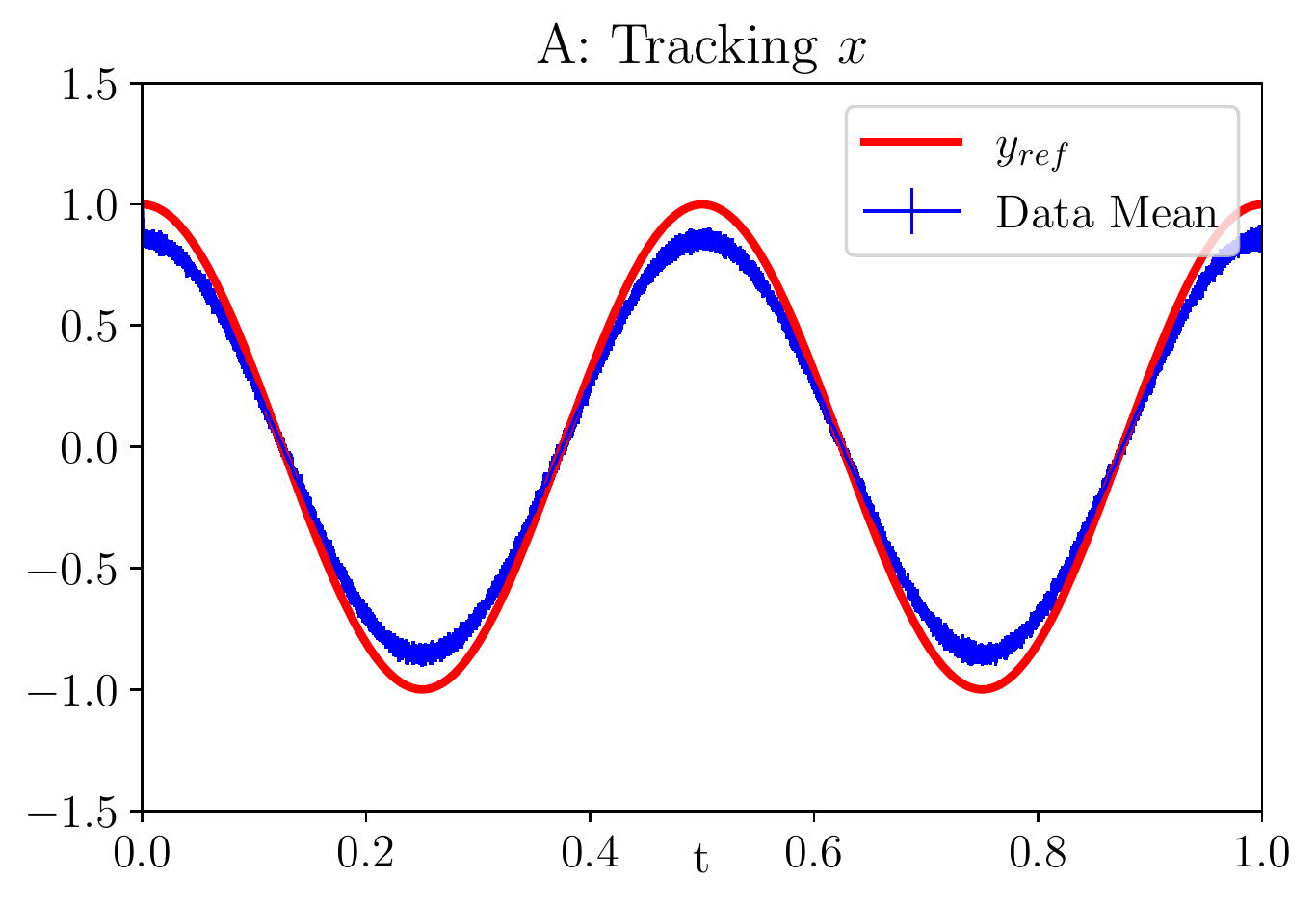}
\includegraphics[width=.45\textwidth]{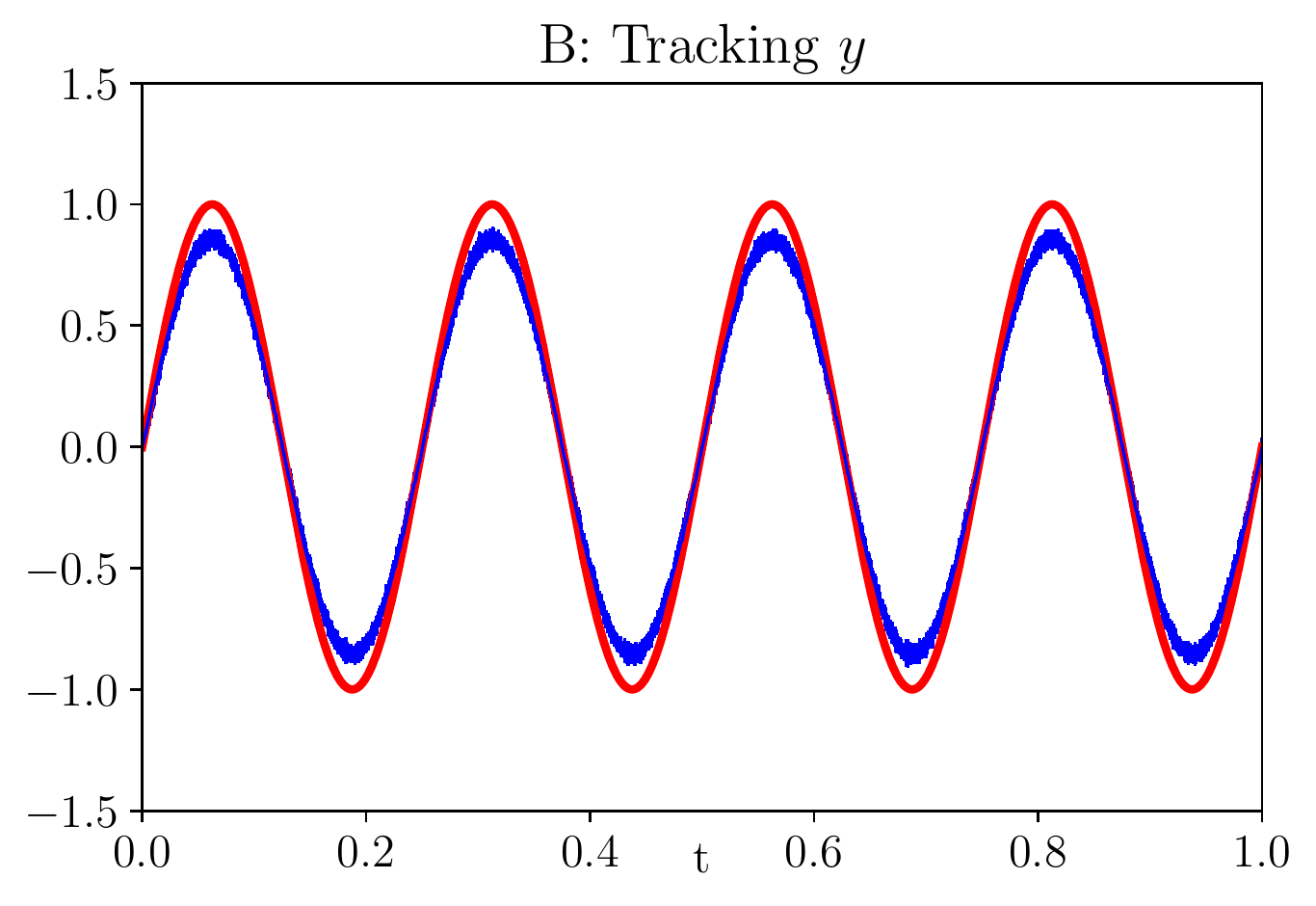}
\includegraphics[width=0.45\textwidth]{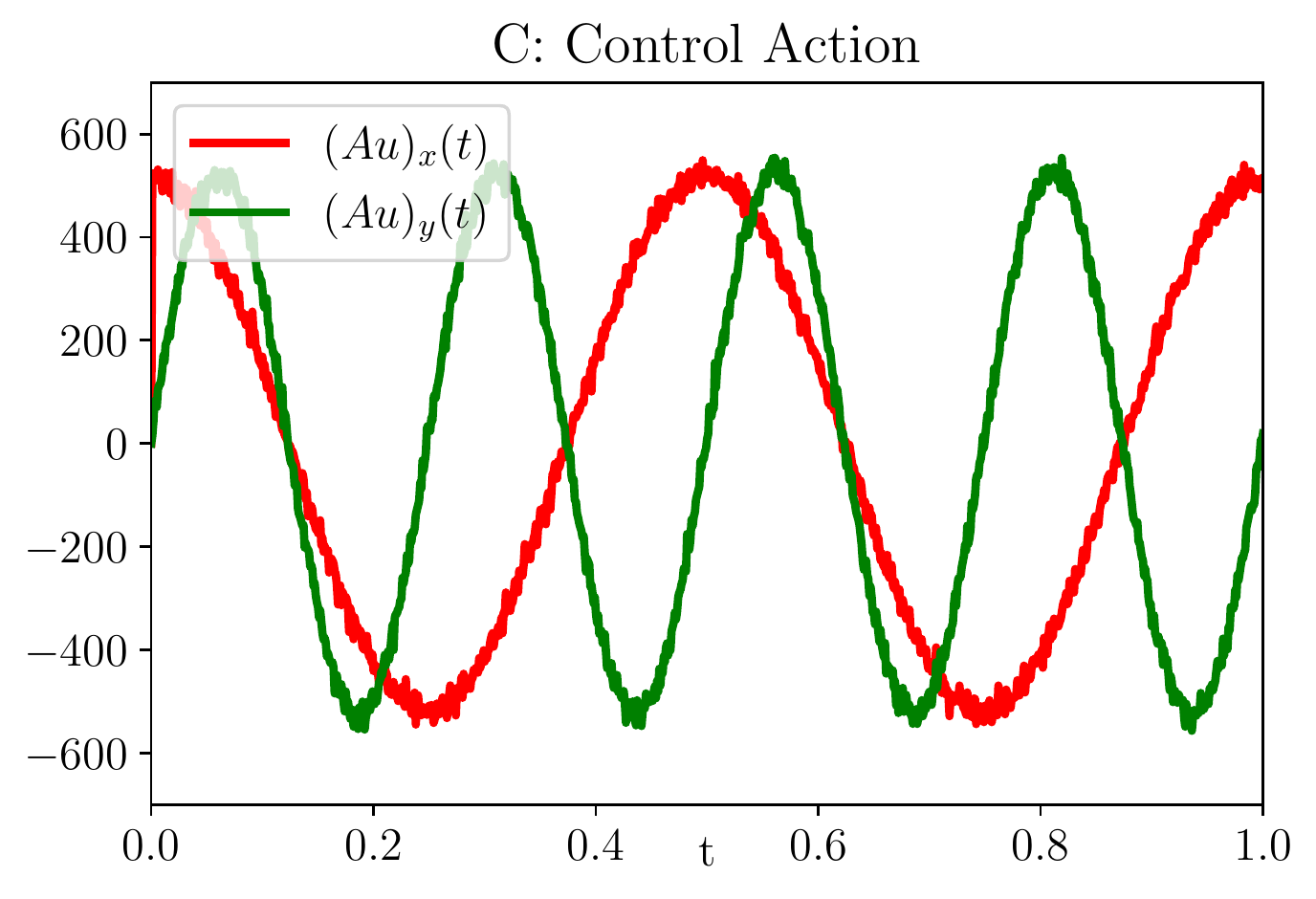}
\includegraphics[width=0.45\textwidth]{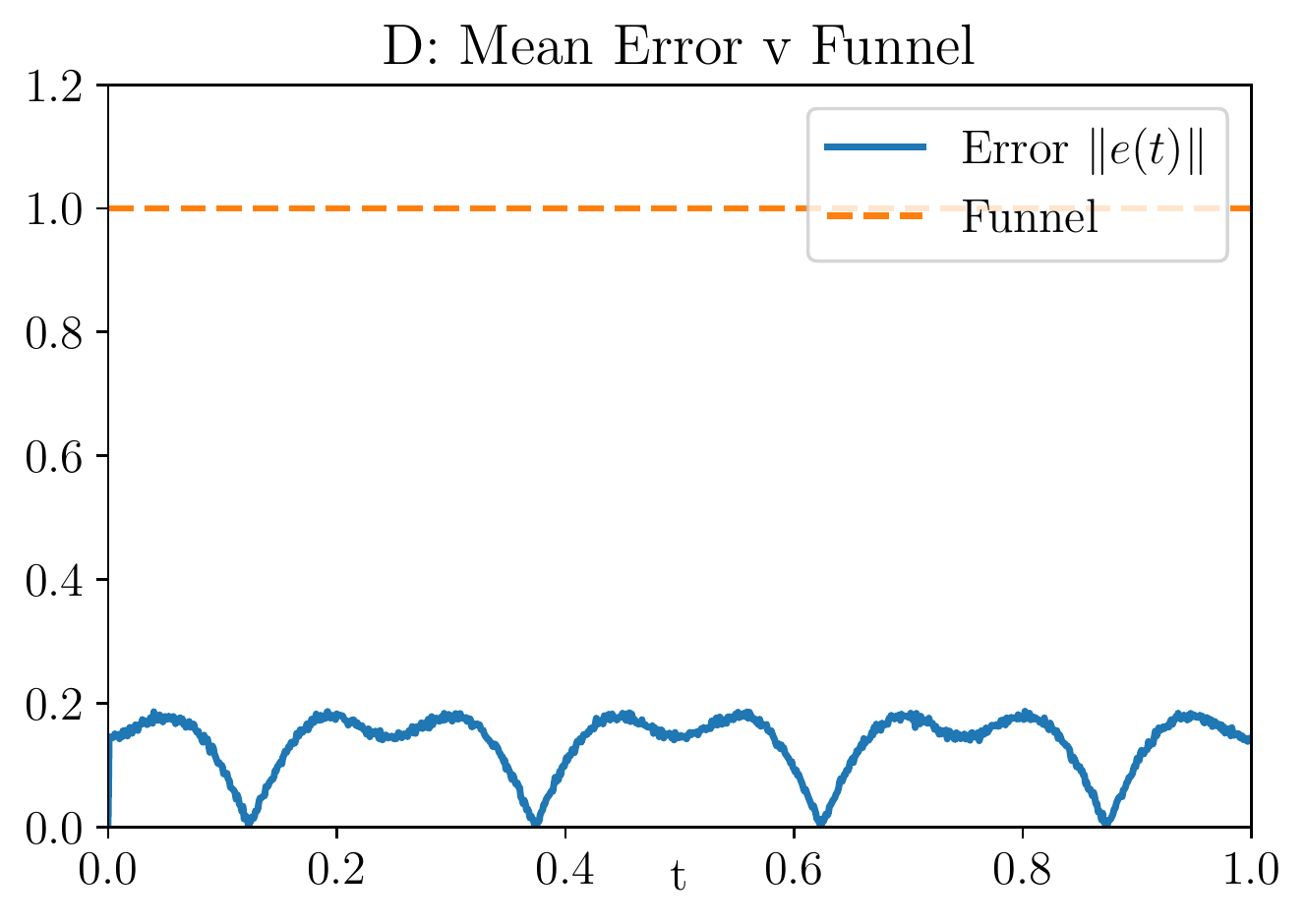}
\caption{(A, B): Comparison of reference signal $y_{\mathrm{ref}}(t)$ (red) and empirical mean value (blue), estimated from $20$ independent simulations. The width of the blue line represents the standard error over all $20$ simulations. Panel A is for the $x$-coordinate, B for the $y$-coordinate. (C): Control action $Au(t)$ for $x$-coordinate (red) and $y$-coordinate (green). (D) Norm of the error $e(t)$ between the reference signal and the empirical mean vector as a function of time. \label{fig:dw_tracking}}
\end{figure}

\begin{figure}
\includegraphics[width=.45\textwidth]{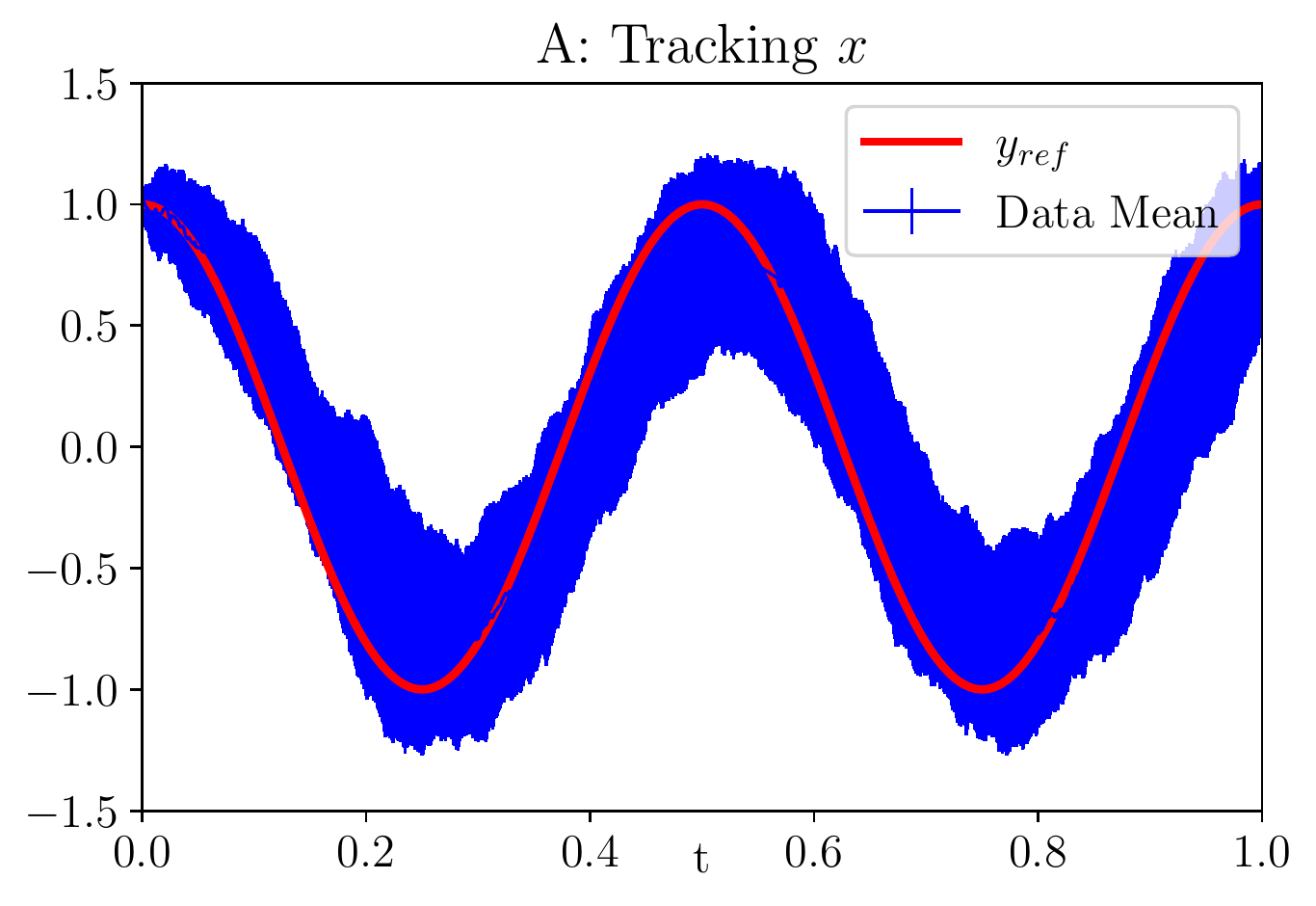}
\includegraphics[width=.45\textwidth]{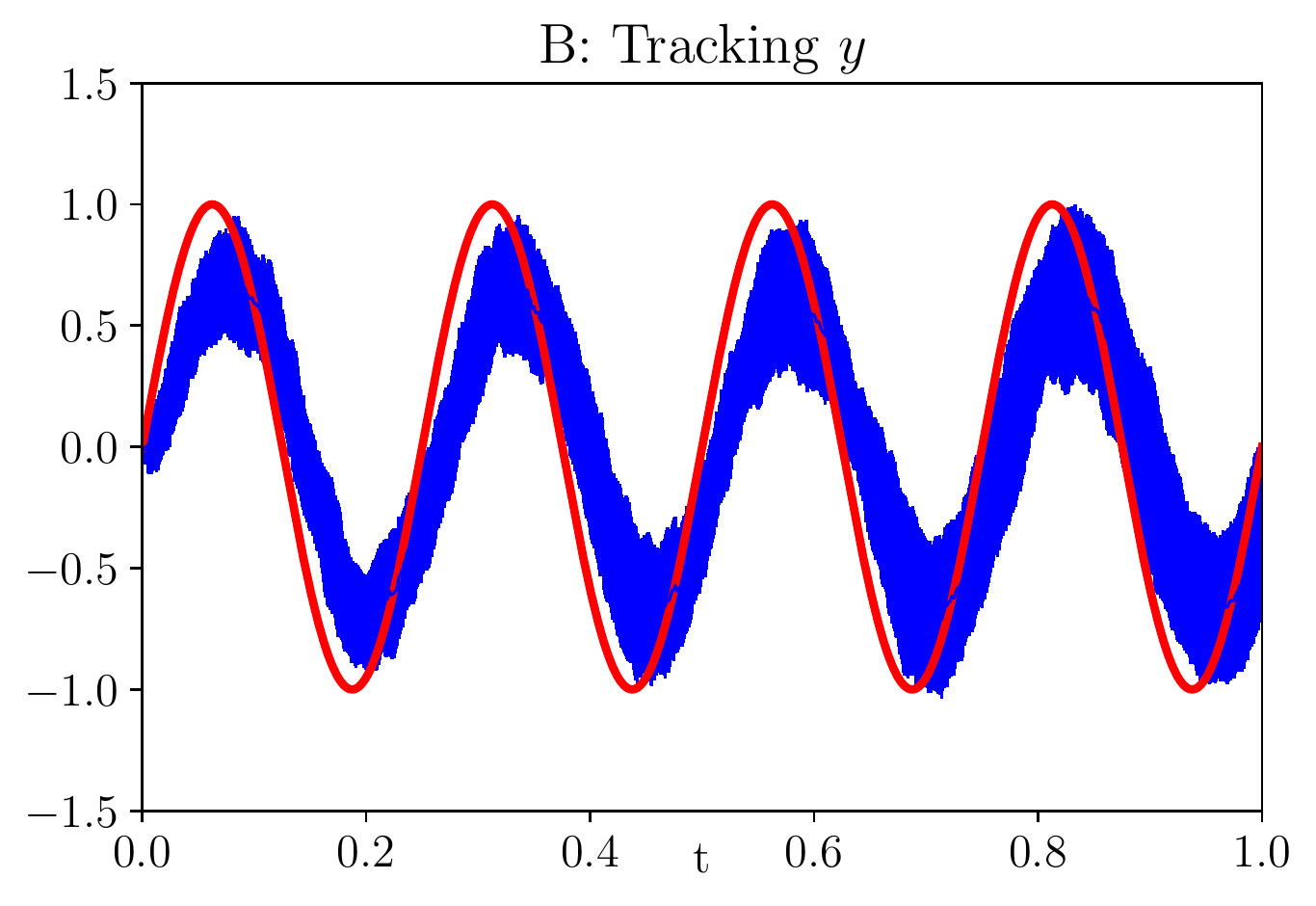}
\includegraphics[width=0.45\textwidth]{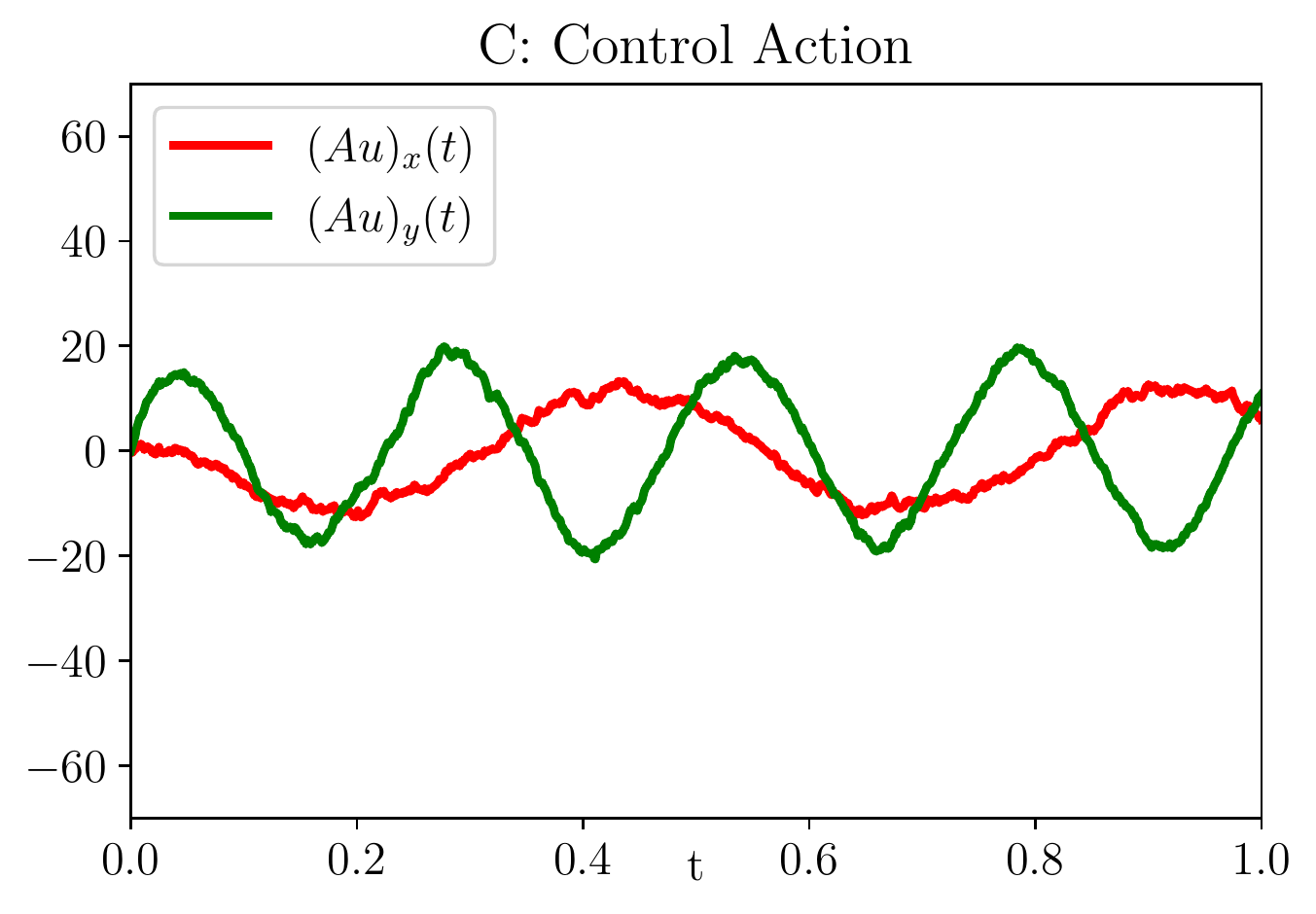}
\includegraphics[width=0.45\textwidth]{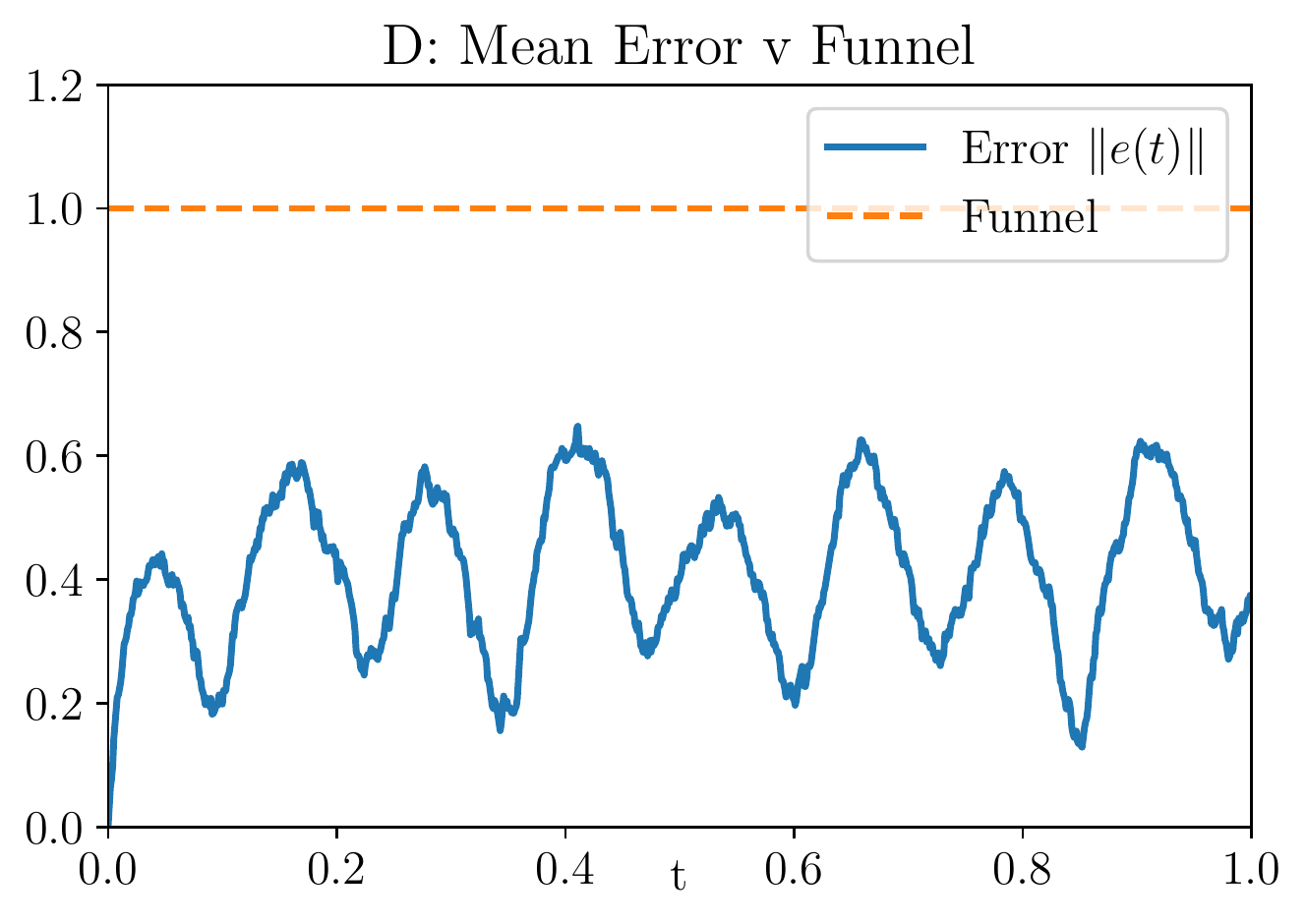}
\caption{Same content as Fig.~\ref{fig:dw_tracking}, but for control strength $a = 5.0$. Note that the scale in panel C is different for the purpose of visualization. \label{fig:dw_tracking_min}}
\end{figure}

\section{Conclusion}
\label{sec:Conclusion}

In the present paper we proposed a new conceptual approach to the sampling problem of SDEs of Langevin type, which is based on the solution of tracking problems using funnel control. We have derived structural conditions on the potential energy which guarantee that funnel control is feasible, and the evolution of the tracking error for the mean values will remain within a prescribed performance funnel. The numerical example of a double-well potential illustrates these theoretical findings, and shows that excellent tracking performance can be achieved using the parameter setting certified by our main result Theorem~\ref{Thm:main}.

However, we have also seen that verification of the theoretical conditions can be quite tedious already for simple potentials. Moreover, the range of certified parameter settings turned out to be quite narrow for the double-well example, while satisfactory performance could also be shown to be possible outside the certified regime. Future research will therefore concentrate on deriving less restrictive conditions. Moreover, the use of output functions different from the mean value, as well as leveraging the capabilities of funnel control for the purpose of enhanced sampling, will be topics of future research.


\bibliographystyle{elsarticle-harv}
\bibliography{Funnel_SDE}

\end{document}